\documentclass[12pt]{article}
\usepackage{amsmath}
\usepackage{amsfonts}
\usepackage{amsthm}
\usepackage{verbatim}
\usepackage{subfig}
\usepackage{graphicx}
\usepackage{cleveref}

\usepackage{amssymb}

\def\R{{ \mathbb{R}}}

\newtheorem{theorem}{Theorem}[section]
\newtheorem{lemma}[theorem]{Lemma}
\newtheorem{proposition}[theorem]{Proposition}

\theoremstyle{definition}

\usepackage{color}

\newtheorem{remark}[theorem]{Remark}

\newtheorem*{remark*}{Remark}
\usepackage{epstopdf}
\usepackage{float}
\numberwithin{figure}{section} \numberwithin{equation}{section}
\title{}

\author{}

\begin{document}
\date{}
\title{Populations with individual variation in dispersal in heterogeneous environments: dynamics and competition with simply diffusing populations}

\author{Robert Stephen Cantrell,$^{1,2}$
Chris Cosner,$^{1,2}$  and
Xiao Yu$^{1,2,3}$}
\date{}
\maketitle

\noindent 1. Department of Mathematics, University of Miami, Coral Gables, FL 33146, USA

\noindent 2. Research partially supported by NSF Grant DMS-1514752

\noindent 3. School of Mathematical Sciences, South China Normal University, Guangzhou, Guangdong, 510631, China(Email: xymath19@m.scnu.edu.cn)

\vspace{.50in}



\newpage

\begin {abstract} We consider a model for a population in a heterogeneous environment, with logistic-type local population dynamics, under the assumption that  individuals can switch between two different nonzero rates of diffusion.  Such switching behavior has been observed in some natural systems. We study how environmental heterogeneity and the rates of switching and diffusion affect the persistence of the population.  The  reaction-diffusion systems in the models can be cooperative at some population densities and competitive at others. The results  extend our previous work on similar models in homogeneous environments.  We  also consider  competition between two populations that are ecologically identical, but where one population diffuses at a fixed rate and the other switches between two different diffusion rates.  The motivation for that is to gain insight into when switching might be advantageous versus diffusing at a fixed rate. This is a variation on the classical results for ecologically identical competitors with differing fixed diffusion rates, where it is well known that ``the slower diffuser wins".

\end{abstract}

\section{Introduction}
The problem of understanding how dispersal patterns affect population interactions and thus are subject to evolutionarily selection has generated much interest among mathematical biologists.  Classical models for dispersal typically assume that any given type of organism will disperse according to a single pattern or strategy, which may or may not be conditional on environmental conditions.  Various models of that type are discussed in \cite{CCbook,Cosner}.  One specific line of inquiry that has generated significant interest is the problem of deciding which types of dispersal, if any,  are advantageous.   A well known result in that direction is that in environments that vary in space but not in time, if populations that are ecologically identical except for their dispersal pattern compete, and the populations diffuse at different rates, the slower diffuser wins \cite{Hastings, Dockery}.  However, there is considerable evidence that many organisms can switch between different dispersal modes depending on whether they are searching for resources or exploiting them; see \cite{Fleming, Fryxell, Newlands, Prevedello, Rodriguez, Skalski, Ward}. Models that capture the idea of switching between movement modes are developed in \cite{Fagan, Gilliam, Tyson}.  In \cite{CCY} we developed basic theory for a model where a population consists of two sub-populations that diffuse at different rates, individuals can switch between sub-populations,  and where there is logistic-type  self-limitation. Somewhat similar types of models have been proposed in a related but different context, where a population has sub-populations that have different dispersal rates and perhaps different population dynamics and each sub-population is subject to mutations that produce offspring that belong to other sub-populations.  This idea was already discussed in \cite{Dockery}. It has been used to study how dispersal polymorphism can affect the spreading speed of biological invasions\cite{EC,MBC}.  Some very strong and interesting results on traveling waves, spreading speeds, and dynamics for Fisher-KPP models with switching or mutation are presented in  \cite{G1, G2, G3}.  Existence results for equilibria of some related systems on bounded domains are derived in \cite{B, HW}.

 The models for  populations where individuals can switch between two  sub-populations that  we considered in \cite{CCY} and will use here turn out to potentially be  cooperative systems at some densities and  competitive ones at others. Roughly speaking, when switching rates are high, the models are asymptotically cooperative while if switching rates are low they are asymptotically competitive. The version of the model treated in \cite{CCY} had constant coefficients.  In the present paper we will extend some of the results of  \cite{CCY}  to cases  where some coefficients can vary in space.  We will also consider a model for competition on a bounded domain  between a population whose members can switch between two diffusion rates and an otherwise ecologically identical population whose members diffuse at a single intermediate rate.  This is motivated by previous work from the viewpoint of \cite{Dockery} on the evolution of slow diffusion in systems where each competing population has  a single fixed diffusion rate.  See \cite{HN} for more recent results that give a more complete treatment of the case of two competing populations with fixed diffusion rates. We are primarily interested in extending results such as those in  \cite {Dockery, HN}  on how diffusion rates influence competitive interactions to the case where one of the competitors switches between two diffusion rates.  We have chosen to  follow their modeling assumptions and  use no-flux boundary conditions, which are Neumann boundary conditions in our models. The reason for that choice is that  with Dirichlet or Robin boundary conditions, increasing the diffusion rate causes a loss of population across the domain boundary as well as causing different movement patterns in the interior, so it is clear that faster diffusion will be a disadvantage.  However, in the Neumann case, there is no boundary loss so it is very interesting that faster diffusion may still be a disadvantage.  In \cite{CCY} we considered both Neumann and Dirichlet boundary conditions in the case of a single population that could switch between two different diffusion rates.  We found that many of the general abstract results about the models  with Dirichlet conditions were similar to those for models with Neumann conditions, but  there were some differences in more refined specific results, and in some cases Dirichlet conditions caused additional technical difficulties that limited what we could do.  See \cite{CCY} for details.   It would be interesting to consider Dirichlet boundary conditions in the  of models we study in this paper.  That would present some challenges but based on the analysis in \cite{CCY} it should be possible to make some progress.  More generally we think that extending the theory for models with switching to cover a broader range of  dispersal operators, boundary conditions, and population interactions  is an interesting topic for future research.

 It turns out that in our model the result of the competition between the populations with and without switching depends on the relative sizes of the diffusion coefficients  and  on the rates  of switching between faster and slower diffusion by the population that uses two distinct movement modes. In studying competition between populations with and without switching, we assume that the system describing the switching competitor is asymptotically cooperative, so that the full system is eventually cooperative-cooperative-competitive.  This type of system was considered in the case of ordinary differential equations in \cite{Smith1, Smith2}.  It is monotone with respect to  the ordering given by $(u_1,v_1,w_1) \ge (u_2,v_2,w_2) \iff u_1 \ge u_2, v_1 \ge v_2, w_1\le w_2$.  The main methods we will use are primarily monotone dynamical systems theory,  positive operator theory (specifically the Krein-Rutman theorem), and estimates of principal eigenvalues.
 
 The paper is organized as follows. In the  second and third sections, we derive various results on the two-component subsystem describing the population that can switch movement modes.   Some of these are extensions of results from \cite{CCY} to systems where some coefficients vary in space.  In the fourth section we analyze  the stability of semi-trivial equilibria in the full model and give conditions where one or the other of the competing populations will be excluded.  In the fifth section we examine how the stability of the semi-trivial equilibria depends on the switching rates.    We finish with a summary of the conclusions from the analysis.
 
\section{Stability analysis of semi-trivial steady states}
Consider the  system 
\begin{equation}\label{PDEModel}
\aligned
&\frac{\partial u}{\partial t}=d_1\Delta u-\alpha(x) u+\beta (x) v+u(m(x)-u-bv) && \mbox{in}\, (0,\infty)\times \Omega,\\
&\frac{\partial v}{\partial t}=d_2 \Delta v+\alpha(x) u-\beta(x) v+v(m(x)-cu-v)&& \mbox{in}\, (0,\infty)\times\Omega,\\
&\frac{\partial u}{\partial n}=\frac{\partial v}{\partial n}=0 &&\mbox{on}\, (0,\infty)\times\partial \Omega,\\
&u(0,x)=\phi_1(x),\, v(0,x)=\phi_2(x) &&\mbox{in}\, \,\Omega,\\ 
\endaligned
\end{equation}
where $\Omega\subset \R^N$ $(N\ge1)$ is a bounded domain with boundary $\partial \Omega$ of class $C^{2+\theta}$ $(0<\theta\le1)$, $\frac{\partial}{\partial n}$ denotes the differentiation in the direction of outward normal $n$ to $\partial \Omega$.  In general, we will suppose that $0<d_1\le d_2$ and $\alpha,\beta,m\in C^\nu(\bar\Omega)(0<\nu<1)$, $\alpha(x)$ and $\beta(x)$ are non-negative and both positive for some $x_0\in\bar{\Omega}$. We will also assume that $m(x)$   is positive for some $x_1\in \bar{\Omega}$, but we will consider some cases where $m(x)$ changes sign and others where $m(x)$ is positive. The system \eqref{PDEModel} describes the dispersal and population dynamics of a single species that is divided into two groups, for example individuals that are seeking resources and other individuals who have found resources and are exploiting them, and where individuals can switch between groups.  The corresponding model with constant coefficients was studied in \cite{CCY}.  In this section  we will extend some of the ideas and results of \cite{CCY} to cases with variable coefficients. 

The local existence of classical solutions follows from standard
results, see for example the discussion and references in \cite[Sections 1.6.5 and 1.6.6]{CCbook}. Global existence follows if solutions are bounded by some finite $B(T)$ in $[L^\infty(\Omega)]^2$ on any finite time interval $(0,T)$ with $T>0$. Let \begin{equation}\label{fg}
\aligned
g_1(x,u,v)=(m(x)-\alpha(x)-u)u+(\beta(x)-bu)v, \\
g_2(x,u,v)=(m(x)-\beta(x)-v)v+(\alpha(x)-cv)u.
\endaligned
\end{equation}Clearly, there exist $M^+, N^+>0$ such that 
$g_1(x,M^+,v)<0$  and $g_2(x,u,N^+)<0$  for any $(x,u,v)\in\bar{\Omega}\times[0,M^+]\times[0,N^+]$. For such $(x,u,v)$, we have $g_1(x,u,v)\ge u(m-\alpha-bN^+-u)+\beta(x)v$, and $g_2(x,u,v)\ge v(m-\beta-cM^+-v)+\alpha(x)u$. The comparison principle for a scalar parabolic equation applied to each of the equations in \eqref{PDEModel} implies that for any nonnegative and nontrivial initial data, the solution of system \eqref{PDEModel} will stay positive for any $t>0$.  
Indeed, we have the following result on the uniform boundedness of the solution. 
\begin{proposition}\label{Contracting}
	There exist positive numbers $B_1$ and $B_2$, such that for any $M\ge B_1$ and $N\geq B_2$, the rectangular region $[0,M]\times[0,N]$ is invariant and attracting from above,  that is,  $g_1(x,0,v)\geq 0, g_2(x,u,0)\geq 0$, $g_1(x,M,v)<0$ and $g_2(x,u,N)<0$, for any $(x,u,v)\in\bar{\Omega}\times[0,M]\times[0,N]$. Thus, any solution of \eqref{PDEModel} with nonnegative bounded initial data  exists for all $t\ge0$, and eventually lies in the rectangular region $[0,B_1]\times[0,B_2]$. Moreover, if there exist positive numbers $A_1$ and $A_2$ such that $g_1(x,A_1,v)>0$ and $g_2(x,u,A_2)>0$ for any $(x,u,v)\in\bar\Omega\times[A_1,B_1]\times[A_2,B_2]$, then any nontrivial solution of \eqref{PDEModel} with nonnegative bounded initial data  eventually lies in the rectangular region $[A_1,B_1]\times[A_2,B_2]$.
	\end{proposition}

\begin{proof}
	We only show the second part of the proof. Let $$g^-_i(u_1,u_2)=\inf\{g_i(x,\theta_1,\theta_2), \theta_i=u_i, (x,\theta_j)\in\bar\Omega\times[u_j,B_j], j\not=i \}, i=1,2.$$
	Then $g^-_i(u_1,u_2),i=1,2$ is Lipschitz continuous in $[0, B_1]\times[0,B_2]$ and $g^-_i(u_1,u_2)$ is nondecreasing with $u_j, j\not=i, i=1,2$. Thus, the ODE system $\displaystyle\frac{du_i}{dt}=g^-_i(u_1,u_2), i=1,2.$ is a cooperative system. By our assumption, we see that $$g^-_1(A_1,v)\ge g^-_1(A_1,A_2)= \inf\{g_1(x,A_1,v),(x,v)\in\bar{\Omega}\times[A_2,B_2]\}>0$$ for any $v\in[A_2,B_2]$  and $g^-_2(u,A_2)>0$ for any $u\in[A_1,B_1]$.
	 Thus $(A_1,A_2)$ is a strict lower solution for the ODE system. Note that $g=(g_1,g_2)$ is subhomogeneous, in the sense that, for any $\gamma\in(0,1]$, $g_i(x,\gamma u_1,\gamma u_2)\ge \gamma g_i(x,u_1,u_2)$, $(u_1,u_2)\in[0,B_1]\times[0,B_2], i=1,2,$ and thus so is $g^-_i,i=1,2$. One can show that for any $\gamma\in(0,1]$, $[\gamma A_1, B_1]\times[\gamma A_2, B_2]$ is contracting from below for the ODE system. Indeed, $g^-_1(\gamma A_1,v)\ge g^-_1(\gamma A_1,\gamma A_2)\ge\gamma g^-_1(A_1,A_2)>0$ for any $v\in[\gamma A_2,B_2]$. Let $U(t,x,\phi)$ be a solution of system \eqref{PDEModel} with $U(0,x,\phi)=(\phi_1,\phi_2)\in(0,B_1]\times(0,B_2]$. Then there exists $\gamma_0\in(0,1)$ such that $\phi_i\ge\gamma_0 A_i>0, i=1,2.$  Let $U_-(t)$ be the solution of $U_t=G^-(U)=(g^-_1(u_1,u_2),g^-_2(u_1,u_2))$ with $U_-(0)=(\gamma_0 A_1,\gamma_0 A_2)$.  It follows from \cite[Theorem 1]{Smoller1}  that $U(t,x,\phi)\ge U_-(t)$ for any $t\ge0$. Indeed, $U^-(t)$ is nondecreasing in $t$ and bounded from above  and thus converges to some positive point.  Since $\lim\limits_{t\to\infty} U_-(t)\ge(A_1+\epsilon,A_2+\epsilon)$ for some small $\epsilon>0$, $U(t,x,\phi)$ will eventually lie in $[A_1,B_1]\times[A_2,B_2]$.
	 \end{proof}

Set $\underline F:= \min_{x\in\bar\Omega}F(x)$ and $\bar{F}:=\max_{x\in\bar\Omega}F(x)$.    Based on the preceding  observations, we obtain the following result.

\begin{proposition}\label{Cases}The following statements are valid.
	\begin{enumerate}
		\item Assume that $\underline{m}$, $\underline{\alpha}$, $\underline{\beta}>0$ and let $ k=\min\{\underline{\alpha}/\bar \alpha,\underline{\beta}/\bar \beta\}\le 1$. If $\displaystyle k>\max\{1-\underline{m}/(b\bar{m}),1-\underline{m}/(c\bar{m})\}$ and $(\bar{\beta},\bar{\alpha})\in S_1:=\{(x,y): \underline{m}+b(k-1)\bar{m}-y-x/b>0,\underline{m}+c(k-1)\bar{m}-x-y/c>0, x>0,y>0\}$, then any solution of system \eqref{PDEModel} with positive bounded initial data eventually lies in $(\bar\beta/b,\bar{m}]\times(\bar\alpha/c,\bar{m}]$, where \eqref{PDEModel} is a competitive system.
		\item Let $k_1:=\max\{\bar\beta/\underline\beta, \bar\alpha/\underline\alpha\}\ge1$. Assume that 
		$k_1<1+k_0$, where $k_0$ is the larger root of $(bx-c)(cx-b)-1=0$. Then any solution of system \eqref{PDEModel}  with positive bounded initial data eventually lies in $\displaystyle(0,\underline\beta/b)\times(0,\underline\alpha/c)$, where \eqref{PDEModel} is a cooperative system, provided $\displaystyle(\underline\beta/b,\underline\alpha/c)\in S_2:=\{(x,y): \bar{m}-x+(b(k_1-1)-c)y<0, \bar{m}-y+(c(k_1-1)-b)x<0, x>0, y>0\}$.
	\end{enumerate}
\end{proposition}
\begin{proof}
	1. Clearly, if $(\bar{\beta},\bar{\alpha})\in S_1$, then $g_1(x,\bar m,v)<(\bar \beta-b\bar{m})v\le0 $ for any $v\in[0,\bar m]$ and $g_2(x,u,\bar{m})<(\bar \alpha-c\bar m)u\le0$ for any $u\in[0,\bar m]$. Moreover, $g_1(x,\frac{\bar\beta}{b},v)\ge (\underline{m}-\bar{\alpha}-\frac{\bar\beta}{b})\frac{\bar\beta}{b}+(\underline \beta- \bar \beta)\bar m\ge\frac{\bar{\beta}}{b}[\underline m-\bar \alpha-\frac{\bar \beta}{b}+b(k-1)\bar m]>0 $ and $g_2(x,u,\frac{\bar\alpha}{c})\ge(\underline{m}-\bar{\beta}-\frac{\bar\alpha}{c})\frac{\bar\alpha}{c}+(\underline \alpha- \bar \alpha)\bar m\ge\frac{\bar{\alpha}}{c}[\underline m-\bar \beta-\frac{\bar \alpha}{c}+c(k-1)\bar m]>0 $ for any $(x,u,v) \in\bar{\Omega}\times[\frac{\overline\beta}{b},\bar{m}]\times[\frac{\overline\alpha}{c},\bar{m}]$.
	It follows immediately from Proposition \ref{Contracting} that the solution of system \eqref{PDEModel} eventually lies in $(\frac{\overline\beta}{b},\bar{m}]\times(\frac{\overline\alpha}{c},\bar{m}].$
	
	2. For $S_2$, if either $b(k_1-1)-c\le0$ or $c(k_1-1)-b\le0$, that is, $k_1\le1+\frac{c}{b}$ or $k_1\le1+\frac{b}{c}$, then $S_2$ is nonempty. If $b(k_1-1)-c>0$ and $c(k_1-1)-c>0$, as long as $(b(k_1-1)-c)(c(k_1-1)-b)<1$, it is still nonempty. Clearly $k_0>\max\{\frac{b}{c},\frac{c}{b}\}$. Thus if $1\le k_1<1+k_0$, $S_2$ is nonempty.
	 Now for any $\displaystyle(\frac{\underline\beta}{b},\frac{\underline\alpha}{c})\in S_2$, we have $$g_1(x,\frac{\underline{\beta}}{b},v)\le\frac{\underline\beta}{b}[\bar{m}-\underline{\alpha}-\frac{\underline{\beta}}{b}+b(k-1)\frac{\underline\alpha}{c}]=\frac{\underline\beta}{b}\left[\bar{m}-\frac{\underline{\beta}}{b}+(b(k-1)-c)\frac{\underline\alpha}{c}\right]<0$$
	 and 
	$$g_2(x,u,\frac{\underline{\alpha}}{c})\le\frac{\underline\alpha}{c}[\bar{m}-\underline{\beta}-\frac{\underline{\alpha}}{c}+c(k-1)\frac{\underline \beta}{b}]=\frac{\underline\alpha}{c}\left[\bar{m}-\frac{\underline{\alpha}}{c}+(c(k-1)-b)\frac{\underline\beta}{b}\right]<0$$
	holds for $(x,u,v)\in\bar\Omega\times[0,\frac{\underline\beta}{b}]\times[0,\frac{\underline\alpha}{c}]$. The  result then follows.
\end{proof}
Throughout this paper, denote by $\lambda(d,e)$ the principal eigenvalue of 
$$\lambda\phi=d\Delta\phi+e (x)\phi \quad \text{in}\,\, \Omega,\quad \frac{\partial\phi}{\partial n}=0 \quad\text{on}\,\, \partial\Omega.$$
Here $e\in L^\infty(\Omega)$. We collect some useful information on $\lambda(d,e)$; refer to \cite{CC1987,CCbook,HN, Hess} and the references therein.
\begin{proposition}\label{PE}
	\begin{enumerate}
		\item [(a)] $\lambda(d,e)$ depends smoothly on $d>0$; depends continuously on $e\in L^\infty(\Omega)$.
		\item[(b)] If $e_1,e_2\in L^\infty(\Omega)$ and $e_1(x)\ge e_2(x) $ in $\Omega$, then $\lambda(d,e_1)\ge\lambda(d,e_2)$ with equality holding if and only if $e_1\equiv e_2$ a.e. in $\Omega$. 
		\item[(c)]If $e$ is non-constant, then $\lambda(d,e)$ is strictly decreasing in $d>0$.
		\item[(d)] Assume that $e$ is non-constant and changes sign. Then
		\begin{enumerate}
			\item [(i)]If $\int_{\Omega}e\ge0$, then $\lambda(d,e)>0$.
			\item[(ii)]If $\int_{\Omega}e<0$, then there exists a unique $\mu^*>0$ independent of $d$ such that, sign$(1-d\mu^*)$=sign $(\lambda(d,e))$.
		\end{enumerate}
		
	\end{enumerate}

\end{proposition}
 By the celebrated Krein-Rutman Theorem, the eigenvalue problem \begin{equation}\label{EP}
 \aligned
&\lambda \phi_1=d_1\Delta \phi_1-\alpha(x)\phi_1+\beta(x)\phi_2+m(x)\phi_1&&\mbox{in}\,\, \Omega,\\
&\lambda\phi_2=d_2 \Delta \phi_2+\alpha(x) \phi_1-\beta(x)\phi_2+m(x)\phi_2&&\mbox{in}\,\, \Omega,\\
& \frac{\partial \phi_1}{\partial n}=\frac{\partial \phi_2}{\partial n}=0  &&\mbox{on}\,\, \partial\Omega 
\endaligned
\end{equation}
admits a principal eigenvalue $\lambda_0$ with a positive eigenfunction $\psi=(\psi_1,\psi_2)$. See \cite{Amann, Lam,Gomez}. Clearly, if $m>0$ on $\bar\Omega$, then $\lambda_0>0$; if $m<0$ on $\bar\Omega$, then $\lambda_0<0$.  An interesting question  arises when $m$ changes sign. In that case,  what kind of sufficient conditions will guarantee that $\lambda_0>0$, so that $0$ is linearly unstable?   Next, we explore some sufficient conditions through some simple investigation.

\begin{proposition}\label{sc}Assume that $m$ changes sign.  Then the following statements are valid.
	\begin{enumerate}
		\item [(i)]If $\max\{\lambda(d_1,m-\alpha),\lambda(d_2,m-\beta)\}\ge0$, then $\lambda_0>0$;
		\item[(ii)] 	If $\int_{\Omega}m\ge\frac{\int_{\Omega}(\sqrt{\alpha}-\sqrt{\beta})^2}{2}$, then $\lambda_0>0$.
	\end{enumerate}
\end{proposition}
\begin{proof}
	Observe that $(\lambda_0 I-L_1)\phi=\beta\psi_2\ge0\not\equiv0$ in $\Omega$ has a unique positive solution $\psi_1$, where $L_1\phi:=d_1\Delta \phi_1+(m(x)-\alpha(x))\phi_1$ with zero Neumann boundary condition. This yields that $\lambda_0>s(L_1)=\lambda(d_1,m-\alpha)$. Similarly, we have $\lambda_0>\lambda(d_2,m-\beta)$. Consequently, statement (i) holds true.
	
	Note that the components of the positive eigenfunction $\psi=(\psi_1,\psi_2)$ associated with $\lambda_0$ can not both be constant. Otherwise, adding equations of \eqref{EP} together, we obtain that $\lambda_0=m(x)$, which is impossible. Now dividing equations of \eqref{EP} by $\psi_i, i=1,2$ and integrating over $\Omega$, respectively, we have 
	\begin{align*}
	2\lambda_0&=d_1\int_{\Omega}\left|\frac{\triangledown \psi_1}{\psi_1}\right|^2+d_2\int_{\Omega}\left|\frac{\triangledown \psi_2}{\psi_2}\right|^2+2\int_{\Omega}m-\int_{\Omega}\alpha-\int_{\Omega}\beta
		+\int_{\Omega}\left(\beta\frac{\psi_2}{\psi_1}+\alpha\frac{\psi_1}{\psi_2}\right)\\&>2\int_{\Omega}m-\int_{\Omega}\alpha-\int_{\Omega}\beta+2\int_{\Omega}\sqrt{\alpha\beta}\ge2\int_{\Omega}m-\int_{\Omega}(\sqrt\alpha-\sqrt\beta)^2\ge0.
		\end{align*}\end{proof} 
We can also examine our eigenvalue problem \eqref{EP},  by inserting a parameter $\mu$ multiplying $m$ and considering how the  principal eigenvalue  depends on $\mu$. Let $\lambda(\mu), \mu\in\R $ be the principal eigenvalue of  
\begin{equation}\label{EP1}
\aligned
&\lambda \phi_1=d_1\Delta \phi_1-\alpha(x)\phi_1+\beta(x)\phi_2+\mu m(x)\phi_1&&\mbox{in}\,\, \Omega,\\
&\lambda\phi_2=d_2 \Delta \phi_2+\alpha(x) \phi_1-\beta(x)\phi_2+\mu m(x)\phi_2&&\mbox{in}\,\, \Omega,\\
& \frac{\partial \phi_1}{\partial n}=\frac{\partial \phi_2}{\partial n}=0  &&\mbox{on}\,\, \partial\Omega. 
\endaligned
\end{equation}
It is easy to see that the principal eigenvalue $\lambda(0)$ of \eqref{EP1} is zero  with a positive eigenfunction $(\Phi^*_1,\Phi^*_2)$.   The principal eigenvalue $\lambda(\mu)$ is always simple and isolated by Theorem  4.1 of \cite{Gomez}, so it is analytic in $\mu$  by results from  Ch. 7, section 1 and Ch.2, section 1 of \cite{Katobook}.  The operator on the right side of side of  \eqref{EP1}  has a positive resolvent, so  $\lambda(\mu)$  is convex in $\mu$ by results of \cite{Kato82}.  Let $\tilde{\lambda}(\mu)$ and $\tilde{\phi}>0$ be the principal eigenvalue and eigenfunction for
$$\lambda \phi_1=d_1\Delta \phi-\alpha(x)\phi_1+\mu m(x)\phi \quad \mbox{in}\,\, \Omega,\,\, \frac{\partial \phi}{\partial n}=0\quad \mbox{on}\,\, \partial\Omega.$$
Since we assume that $m(x)>0$ for some $x$, Lemma 15.4 of \cite{Hess} implies that $\tilde{\lambda}(\mu) \to \infty$  as  $\mu \to \infty$. (The  notation of \cite{Hess} switches the roles of 
$\lambda$ and $\mu$ we use in our notation  and puts a minus sign on the eigenvalues corresponding to those we denote by  $\lambda$.) Finally, if we multiply the first equation of \eqref {EP1} by  $\tilde{\phi}$, integrate over  $\Omega$, then use Green's formula and the equation for $\phi$ we obtain
$$[\lambda(\mu)-\tilde{\lambda}(\mu)]\int_{\Omega}\phi_1\tilde{\phi}=\int_{\Omega}\beta \phi_2\tilde{\phi}>0,$$
so that $\lambda(\mu)>\tilde{\lambda}(\mu)$ and hence  $\lambda(\mu) \to \infty$ as $\mu \to \infty$. Alternatively, we can show that $\lambda(\mu)$ and the normalized eigenfunctions associated with it are differentiable by methods similar to those used in the proof of Lemma 1.2 of \cite{CC1987}.   (We show the details of a similar argument later in this paper, in the proof of Lemma 5.1.) We have the following observation:
\begin{proposition}\label{thresh}
	Assume that  $m$ changes sign. Then the following statements are valid. 
	\begin{enumerate}
		\item[(i)] If $\int_{\Omega}m(\Phi^*_1+\Phi^*_2)\ge0$, then $\lambda_0>0$.
		\item[(ii)]If $\int_{\Omega}m(\Phi^*_1+\Phi^*_2)<0$, then there exists a unique positive $\mu^0>0$ such that sign$(1-\mu^0)=$ sign$(\lambda_0)$. 
	\end{enumerate}
\end{proposition}
\begin{proof}
 If $\lambda'(0)>0$, it then easily follows from the convexity of $\lambda(\mu)$ that $\lambda(\mu)>\lambda(0)=0, \forall\mu>0$. If $\lambda'(0)=0$, since $\lambda(\mu)$ is analytic, convex and $\lambda(\infty)=\infty$, we have $\lambda'(\mu)>\lambda'(0)=0$ for $\mu>0$. Thus, $\lambda(\mu)>\lambda(0)=0$. If $\lambda'(0)<0$, then $\lambda(\mu)<0$ for $0<\mu \ll 1$. Note that $\lambda(\infty)=\infty$, we infer that there exists a $\mu^0>0$ such that $\lambda(\mu^0)=0$. Now the convexity and analyticity of $\lambda(\mu)$ yield that $\mu^0$ has to be unique. Moreover, $\lambda(\mu)<0$ when $0<\mu<\mu^0$, and $\lambda(\mu)>0$ when $\mu>\mu^0$.	
 
Next, we compute $\lambda'(0)$. Let $(\tilde{\Phi}_1(x,\mu),\tilde{\Phi}_2(x,\mu))$ be the positive eigenfunction associated with $\lambda(\mu)$. By  arguments similar to those used in the proofs of Lemma 1.2 of \cite{CC1987} and Lemma \ref{beta} of the present paper, we can  differentiate \eqref{EP1} with respect to $\mu$ at $\mu=0$. It then follows that 
\begin{eqnarray}\label{EP3}
&&\lambda'(0) \Phi^*_1=d_1\Delta \tilde{\Phi}_{1\mu}-\alpha(x)\tilde{\Phi}_{1\mu}+\beta(x)\tilde{\Phi}_{2\mu}+ m(x)\Phi^*_1,\nonumber\\
&&\lambda'(0) \Phi^*_2=d_2 \Delta \tilde{\Phi}_{2\mu}+\alpha(x) \tilde{\Phi}_{1\mu}-\beta(x)\tilde{\Phi}_{2\mu}+ m(x)\Phi^*_2,
\end{eqnarray}
where $\tilde{\Phi}_{i\mu}=\frac{\partial \tilde{\Phi}_{i}}{\partial \mu}(x,0), \, i=1,2$.
Adding the above equations together and integrating over $\Omega$, we obtain that
$$\lambda'(0)=\frac{\int_\Omega m(\Phi^*_1+\Phi^*_2)}{\int_\Omega (\Phi^*_1+\Phi^*_2)}.$$ 
\end{proof}
Note that
\begin{eqnarray}\label{EP2}
&&0=d_1\Delta \Phi^*_1-\alpha(x)\Phi^*_1+\beta(x)\Phi^*_2,\nonumber\\
&&0=d_2 \Delta \Phi^*_2+\alpha(x) \Phi^*_1-\beta(x)\Phi^*_2,\\
&& \frac{\partial \Phi^*_1}{\partial n}=\frac{\partial \Phi^*_2}{\partial n}=0\quad \nonumber
\end{eqnarray} implies $\Delta(d_1\Phi^*_1+d_2\Phi^*_2)=0$ in $\Omega$, and $\frac{\partial (d_1\Phi^*_1+d_2\Phi^*_2)}{\partial n}=0$ on $\partial \Omega$. Therefore, $d_1\Phi^*_1+d_2\Phi^*_2=C>0$ for some constant $C$. Then substitute  $\Phi^*_2=\frac{C}{d_2}-\frac{d_1}{d_2}\Phi^*_1$ into the first equation of \eqref{EP2} and integrate over $\Omega$, it gives
$$0=-\int_{\Omega}\alpha\Phi^*_1+\int_{\Omega}\beta\left[\frac{C}{d_2}-\frac{d_1}{d_2}\Phi^*_1\right],$$
and hence, $\displaystyle C=\frac{\int_{\Omega}(d_2\alpha+d_1\beta)\Phi^*_1}{\int_{\Omega}\beta}$.  It now follows that 
$\int_\Omega m(\Phi^*_1+\Phi^*_2)=[1-\frac{d_1}{d_2}]\int_{\Omega}\Phi^*_1m+\frac{C}{d_2}\int_{\Omega}m$. Therefore, sign($\lambda'(0)$) is the same as that of $[1-\frac{d_1}{d_2}]\int_{\Omega}\Phi^*_1m+\frac{C}{d_2}\int_{\Omega}m$.

Suppose that $\alpha(x)=k\beta(x)$ for some constant $k>0$. Then $(\Phi^*_1,\Phi^*_2)$ is constant, in fact we can choose $(\Phi^*_1,\Phi^*_2)=c_0(1,k)$, and as a consequence, Proposition \ref{thresh} gives the following result.
\begin{proposition}\label{ST}
	Assume that  $m$ changes sign and $\alpha(x)=k\beta(x)$ for some constant $k>0$. Then the following statements are valid. 
	\begin{enumerate}
		\item[(i)] If $\int_{\Omega}m\ge0$, then $\lambda_0>0$.
		\item[(ii)]If $\int_{\Omega}m<0$, then there exists a unique positive $\mu^*>0$ such that sign$(1-\mu^*)=$ sign$(\lambda_0)$. 
	\end{enumerate}
\end{proposition}

(This is analogous to the case of a single equation.) Clearly, by Propositions \ref{PE}(d) and \ref{sc}, if, for example, $\int_\Omega (m-\alpha)<0$ and $m-\alpha$ changes sign, when $d_1$ is small, $\lambda_0>0$. This suggests  we could study the effects of diffusion rates on $\lambda_0$ in a more direct way.

 Consider the eigenvalue problem with $d>0$ and $\mu>0$
\begin{equation}\label{EP4}
\aligned
&L\Phi:=d\mathcal{L}\Phi+\mu M(x)\Phi=\lambda \Phi&&\mbox{in}\,\,\Omega,\, \frac{\partial \Phi}{\partial n}=0  &&\mbox{on}\,\,\partial\Omega,
\endaligned
\end{equation}
where $\mathcal{L}\phi=\left(\begin{array}{cc}
\Delta& 0 \\
0& d_0\Delta
\end{array}\right)\left(\begin{array}{c}
\phi_1 \\
\phi_2
\end{array}\right)$ with $\frac{\partial \phi_i}{\partial n}=0$, $i=1,2,$ $d_0>0$ is given and $$M(x)=\left(\begin{array}{cc}
m(x)-\alpha(x)& \beta(x) \\
\alpha(x) & m(x)-\beta(x)
\end{array}\right).$$ 
  We extend our notation to denote the principal eigenvalue of \eqref{EP4} as $\lambda(d,\mu M )$. Note that if $d_0=1$, problem \eqref{EP4} can be reduced to the classical scalar eigenvalue problem $d\Delta \Phi+\mu m(x)\Phi=\lambda \Phi$, in other words, $\lambda(d,\mu M )=\lambda(d,\mu m)$.  Below we only focus on $d_0>1$.

For each given $x\in\bar\Omega$, let $s (M(x))$ be the spectral bound of $M(x)$, which is the largest real eigenvalue  due to the Perron-Frobenius Theorem.   Since $\lambda_1=m(x)$ and $\lambda_2=m(x)-\alpha(x)-\beta(x)$ are the two real eigenvalues of $M(x)$, it easily follows that $s(M(x))=m(x)$. 


\smallskip

 \begin{proposition}\label{FP}
 	Assume that $m$ changes sign and $\int_{\Omega}m<0$. Then when $d=1$, there exist finitely many values of  $\mu>0$ such that $\lambda(1,\mu M)=0$.
 \end{proposition}
 \begin{proof}
 	We first claim that when $d=1$, there exists $\mu_0(d_0)>0$, such that $\lambda(1,\mu)>0$ for any $\mu>\mu_0$.
 		Note that $\lim\limits_{d\to0}\lambda(d,M)=\max\limits_{x\in\bar\Omega}s(M(x))=\max\limits_{x\in\bar\Omega}m(x)>0$ (see, e.g., \cite{Dancer,Lam}; related results on singularly perturbed competition systems are obtained in \cite{Rincon}). So there exists a small $\delta_0>0$ such that $\lambda(d,M)>0$ for $d\in(0,\delta_0)$. Now let $\mu_0=\frac{1}{\delta_0}>0$, it follows from $\lambda(1,\mu M)=\mu\lambda(\frac{1}{\mu},M)$ that $\lambda(1,\mu M)>0$ for any $\mu>\mu_0$.

 		Set $k_0=\frac{\int_{\Omega}\alpha}{\int_{\Omega}\beta}>0$. Let $f_1(\mu,\cdot)$ and $f_2(\mu,\cdot)$ be the  eigenfunctions associated with $\lambda(1,\mu(m-\alpha+k_0\beta))$ and $\lambda(d_0,\mu(m-\beta+\frac{1}{k_0}\alpha))$, such that $f_1(0,\cdot)=1$ and $f_2(0,\cdot)=k_0$. 
 		Note that for any $D>0$, we have $\lambda(D,0)=0$.
		As in the derivation of \eqref{EP3} in the proof of Proposition \ref{thresh}, the eigenvalues $\lambda(1,\mu(m-\alpha+k_0\beta))$ and $\lambda(d_0,\mu(m-\beta+\frac{1}{k_0}\alpha))$ are differentiable with respect to $\mu$. Differentiating with respect to $\mu$, integrating over $\Omega$, and letting $\mu \to 0$, we obtain that
		 $$\frac{\partial\lambda(D,\mu (m-\alpha+k_0\beta))}{\partial \mu}(D,0)=\frac{\partial\lambda(D,\mu (m-\beta+\frac{1}{k_0}\alpha))}{\partial \mu}(D,0)=\frac{1}{|\Omega|}\int m: =A<0.$$
		 
 		Our next goal is to show that if $\mu>0$ is sufficiently small, then there exists $\phi(\mu)=(\phi_1,\phi_2)$ such that $L\phi\ll0$.  If such a $\phi$ exists, it  will be a positive super     
		solution  of  $L\phi=0$ with $L$ as in \eqref{EP4},  which then implies that   $\lambda(1,\mu M)<0$.  (This follows from the characterization of the strong maximum principle in Theorem 13 of \cite{Amann}, which 
		gives an extension of results of \cite{Gomez} to systems with
		general boundary conditions.  The key results of \cite{Amann, Gomez} are that for cooperative systems such as \eqref{EP4}, three things are equivalent: the operator $L$ has a strong maximum principle, the principal eigenvalue is negative, and there exists a strictly positive supersolution.)
 		
 		Denote $\beta_{\max}=\max_{x\in\bar\Omega}\beta(x)$ and $\alpha_{\max}=\max_{x\in\bar\Omega}\alpha(x)$. For any sufficiently small $\epsilon>0$ satisfying 
 		$(A+\epsilon)(1-\epsilon)+(k_0+1)\beta_{\max}\epsilon<0$ and $(A+\epsilon)(k_0-\epsilon)+(\frac{1}{k_0}+1)\alpha_{\max}\epsilon<0$,  there exists $\mu_0>0$, such that for any 
		$0<\mu<\mu_0$, we have $$\|f_1(\mu,\cdot)-1\|_{\infty}<\epsilon,\quad \|f_2(\mu,\cdot)-k_0\|_{\infty}<\epsilon$$
 		and $$\left|\frac{\lambda(1,\mu(m-\alpha+k_0\beta))}{\mu}-A\right|<\epsilon,\quad \left|\frac{\lambda(d_0,\mu(m-\beta+\frac{1}{k_0}\beta))}{\mu}-A\right|<\epsilon$$
 		Let $\phi(\mu)=(f_1(\mu),f_2(\mu))$. Then for $0<\mu<\mu_0$,
 		we have 
 		\begin{align*}
 		&\Delta f_1+\mu(m-\alpha+k_0\beta)f_1+\mu (-k_0\beta f_1+\beta f_2)\\
 		&=\mu\left(\frac{\lambda(1,\mu(m-\alpha+k_0\beta))}{\mu}f_1-k_0\beta f_1+\beta f_2\right)\\
 		&<\mu [(A+\epsilon)(1-\epsilon)-k_0\beta(1-\epsilon)+\beta(k_0+\epsilon)]\\
 	&\le\mu[(A+\epsilon)(1-\epsilon)+(k_0+1)\beta_{\max}\epsilon]<0,
 		\end{align*}
 	and	
 	\begin{align*}
 	&d_0\Delta f_2+\mu(m-\beta+\frac{1}{k_0}\alpha)f_2+\mu \left(-\frac{\alpha}{k_0} f_2+\alpha f_1\right)\\
 	&=\mu\left(\frac{\lambda(d_0,\mu(m-\beta+\alpha/k_0))}{\mu}f_2-\frac{\alpha}{k_0} f_2+\alpha f_1\right)\\
 	&<\mu [(A+\epsilon)(k_0-\epsilon)-\frac{\alpha}{k_0}(k_0-\epsilon)+\alpha(1+\epsilon)]\\
 	&\le\mu[(A+\epsilon)(k_0-\epsilon)+(\frac{1}{k_0}+1)\alpha_{\max}\epsilon]<0.
 	\end{align*}
 	So $L\phi\ll0$, and hence, the characterization of the maximum principle in \cite{Amann} implies that $\lambda(1,\mu M)<0$ for any $0<\mu<\mu_0$. 	\\
 	An alternative approach:
 	 Let $S_\mu(t)$ be the solution semigroup for $U_t=\mathcal{L}U+\mu M(x)U$ on $X:=C(\bar\Omega,\R^2)$. For every $\mu>0$, $S_\mu(t)$ is compact and strongly positive for $t>0$, in view of Krein Rutman Theorem (see \cite[Theorem 7.2]{Hess} and \cite[Theorem 7.6.1]{H}), it follows that the spectral radius $r(S_\mu(t))=e^{\lambda(1,\mu M) t}$ for any $t>0$. 
 	Clearly, for any given $0<\mu<\mu_0$, $1\cdot\phi-S_\mu(t)\phi:=h>0$ in $X$ for $t>0$.
 	
 	Now \cite[Theorem 7.3]{Hess} implies $1>r(S_\mu(t))=e^{\lambda(1,\mu M)t}$ for every $t>0$, and hence, $\lambda(1, \mu M)<0$.

 	As a consequence, we see that equation $\lambda(1,\mu M)=0$ admits at least one positive root. Indeed, the roots of $\lambda(1,\mu M)=0$ are isolated due to the fact that $\lambda$ is analytic in $\mu\in(0,\infty)$ (see, e.g., \cite[Theorem 4.1]{Gomez} and \cite{Katobook}). Thus, there are a finite number of values of $\mu\in(0,\mu_0)$ such that $\lambda(1,\mu M)=0$.    (We cannot give conditions that guarantee there is a unique value of $\mu$, as in the scalar case, because the key lemma derived for that purpose in \cite{HK} is not available  for systems.) \end{proof}
\begin{lemma}
Assume that $m$ changes sign. Let $d_0=\frac{d_2}{d_1}$ be fixed and $d=d_1$ vary.
\begin{enumerate}
	\item [(a)] If $\int_\Omega m<0$, then there exists $0<C_1\le C_2$ dependent on $d_0$ and $M$ such that,
	 \begin{enumerate}
	 	\item [(i)] If $d<C_1$, then $\lambda_0>0$.
	 	\item[(ii)] If $d>C_2$, then $\lambda_0<0$.
	 	\item[(iii)] there are a finite number of $d\in[C_1,C_2]$, such that $\lambda_0=0$. 
	 \end{enumerate} 
 \item [(b)] If $\int_\Omega m>0$, then $\lambda_0>0$ provided $d$ is either large or small.
\end{enumerate}
\end{lemma}
\begin{proof}
	Statement (a) is a direct consequence of Proposition \ref{FP}.
	The proof of Statement (b) is similar to that in Proposition \ref{FP}. Indeed, we can construct $\psi(\mu)\gg0$ such that $L\psi\gg0$ when $\mu$ is sufficiently small. Moreover,
	$(-\psi)-S_\mu(t)(-\psi):=h_1>0$ in $X$. Then
	\cite[Theorem 7.3]{Hess} again implies $1<r(S_\mu(t))=e^{\lambda(1,\mu M)t}$, that is, $\lambda(1,\mu M)>0$ for $\mu$ sufficiently small. 	
\end{proof}
 Because we are unable to show that there is a unique root of $\lambda(1,\mu M)=0$ in Proposition \ref{FP}, a sharper result  for Statement (b) is not available that for arbitrary $d>0$, $\lambda_0>0$.   However when $d_1=d_2$, the result is analogous to a scalar equation, and $\lambda_0$ depends continuously on $d_1$ and $d_2$. Therefore, a perturbation argument implies the following result.
\begin{lemma}
	Assume that $m$ changes sign. Let $d=d_1$ be fixed and $d_0=\frac{d_2}{d_1}$ vary.
	\begin{enumerate}
		\item [(a)]If $\int_\Omega m<0$ and $\mu^*$ is defined in Proposition \ref{PE}, then there exists a small $\delta(d_1)>0$, such that for any $d_0\in(1,1+\delta)$
		\begin{enumerate}
			\item [(i)] If $d_1>\frac{1}{\mu^*}$, then $\lambda_0<0$.
			\item[(ii)] If $d_1<\frac{1}{\mu^*}$, then $\lambda_0>0$.
		\end{enumerate} 
		\item [(b)] If $\int_\Omega m\ge0$,  there exists a small $\delta(d_1)>0$, such that for any $d_0\in(1,1+\delta)$, $\lambda_0>0$. 
	\end{enumerate}
\end{lemma}
\begin{proof}
	(a) (i) When $d_0=1$, there exists $\mu^*>0$ such that $\lambda_0=\lambda(d_1,\mu m)<0$ if and only if $d_1>\frac{1}{\mu^*}$. Now for any given $d_1>\frac{1}{\mu^*}$ and $d_0=1$, we have $\lambda_0<0$.  Since $\lambda_0$ depends continuously on $d_0>0$,  there exists some $\delta(d_1)>0$, such that $\lambda_0<0$ for any $d_0\in(1,1+\delta)$. Similarly, we can verify other cases.
\end{proof}
Now we have the following practical persistence result in terms of $\lambda_0$. Let $X_1=C(\bar \Omega,\R^2)$ and $X^+_1=C(\bar \Omega,\R^2_+)$.
\begin{theorem}\label{UP}
	Let $u(t,x,\phi)$ be the solution of \eqref{PDEModel} with $u(0,\cdot,\phi)=\phi\in X^+_1$.
	\begin{enumerate}
		\item [(i)] If $\lambda_0\leq 0$, then $0$ is globally attractive for any $\phi\in X^+_1$.
		\item [(ii)] If $\lambda_0>0$, then system \eqref{PDEModel} admits at least one positive steady state $(U^*,V^*)$, and there exists an $\eta>0$ such that for any $\phi\in X^+_{1}\setminus\{0\}$, we have
	$$\liminf_{t\to\infty}u_i(t,x,\phi)\ge\eta, \quad\forall i=1,2.$$
	\end{enumerate}
\end{theorem}
\begin{proof}
	(i) It is easy to see that for any $t>0$ 
	\begin{eqnarray}\label{PDEineq}
	&&\frac{\partial u_1}{\partial t}\le d_1\Delta u_1-\alpha(x) u_1+\beta (x) u_2+m(x)u_1,\nonumber\\
	&&\frac{\partial u_2}{\partial t}\le d_2 \Delta u_2+\alpha(x) u_1-\beta(x) u_2+m(x)u_2.\nonumber
	\end{eqnarray}
	Therefore, for any $\phi\in X^+_1$, there exists a number $p>0$, such that $\phi\le p\psi$ where $\psi$ is the positive eigenfunction associated with $\lambda_0$, and hence, the comparison principle (for the linearized system of \eqref{PDEModel}) implies $u(t,\cdot,\phi)\le p e^{\lambda_0 t} \psi$ for any $t\ge0$.  If $\lambda_0<0$, let $t\to \infty$. Then the statement (i) follows for that case. \\
Suppose that $\lambda_0=0$. It follows from the Krein-Rutman theorem that the adjoint of the operator on the right side of \eqref{EP} has a principal eigenvalue equal to $\lambda_0=0$ with a positive eigenfunction. Let $\psi^*=(\psi^*_1,\psi^*_2)$ be a positive eigenfunction for the adjoint problem for  \eqref{EP}.  If we multiply the first equation in \eqref{PDEModel} by $\psi^*_1$ and the second by $\psi^*_2$ and then integrate over $\Omega$ and add the resulting equations, all the terms arising from the linear part of the right side of \eqref{PDEModel} drop out and we obtain

\begin{equation}\label{int1}
 \frac{d }{dt}\int_{\Omega}(\psi^*_1u+\psi^*_2v)=-\int_{\Omega}[\psi^*_1u(u+bv)+\psi^*_2v(cu+v)].
 \end{equation}
 Since $\psi^*_1$ and $\psi^*_2$ are positive and continuous on $\bar\Omega$, they are bounded above and below by positive constants so that 
 $[\psi^*_1u(u+bv)+\psi^*_2v(cu+v)] \geq c_0(\psi^*_1u+\psi^*_2v)^2$ for some positive constant $c_0$.  It then follows from \eqref{int1} and the Cauchy-Schwartz inequality  that 
 
$$\frac{d }{dt}\int_{\Omega}(\psi^*_1u+\psi^*_2v) \leq -c_0\int_{\Omega}(\psi^*_1u+\psi^*_2v)^2 \leq -\frac{c_0}{|\Omega|} \left[\int_{\Omega}(\psi^*_1u+\psi^*_2v)\right]^2$$
so that 
\begin{equation}\label{int2}
\int_{\Omega}(\psi^*_1u+\psi^*_2v) \to 0 \quad\mbox{as}\quad t \to \infty.
 \end{equation}
If $(0,0)$ is not globally attractive, then for some solution $(u,v)$ of \eqref{PDEModel} there must exist a constant
 $\epsilon>0$ and a sequence $\{t_n\}$ with $t_n \to \infty$ as $n\to \infty$ such that $||(u(t_n),v(t_n)||_{X_1}>\epsilon$. All solutions of \eqref{PDEModel} in $X_1^+$ are bounded  by Proposition \ref{Contracting}. Standard results on parabolic regularity and Sobolev embedding then imply that forward orbits are precompact in $X_1$, so there must be a subsequence of 
 $\{(u(t_n),v(t_n))\}$ that converges in  $X_1$.  By re-indexing we can denote this subsequence as  $(u(t_n),v(t_n))$, then $(u(t_n),v(t_n)) \to (u^*,v^*)$ for some $(u^*,v^*)$ as $n \to \infty$, with
 $|| (u(t_n),v(t_n))||_{X_1} \geq \epsilon$  so that $(u^*,v^*) \neq (0,0)$.  For all sufficiently large $n$ we must have $u(t_n) \geq u^*/2$  and $v(t_n) \geq v^*/2$ so that 
 $$\int_{\Omega}(\psi^*_1u(t_n)+\psi^*_2v(t_n)) \geq \frac{1}{2} \int_{\Omega}(\psi^*_1u*+\psi^*_2v*) >0.$$ SInce $t_n \to \infty$ as $n\to \infty$, this contradicts \eqref{int2}.  To avoid a  contradiction we must have $(0,0)$ globally attractive.

	(ii) Since $\lambda_0>0$, there exists small $\epsilon>0$ such that  the perturbed eigenvalue problem
\begin{equation}\label{PEP}
\aligned
&\lambda \phi_1=d_1\Delta \phi_1-\alpha(x)\phi_1+\beta(x)\phi_2+(m(x)-2\epsilon)\phi_1&&\mbox{in}\,\, \Omega,\\
&\lambda\phi_2=d_2 \Delta \phi_2+\alpha(x) \phi_1-\beta(x)\phi_2+(m(x)-2\epsilon)\phi_2&&\mbox{in}\,\, \Omega,\\
& \frac{\partial \phi_1}{\partial n}=\frac{\partial \phi_2}{\partial n}=0  &&\mbox{on}\,\, \partial\Omega 
\endaligned
\end{equation}
	admits a positive principal eigenvalue $\lambda^\epsilon_0$ with a positive eigenfunction $\psi^\epsilon$.
	
	Let $\mathbb W:=\{\phi\in X^+_1:\phi\not\equiv0\}$ and $\partial \mathbb W:=\{\phi\in X^+_1:\phi\equiv0\}$. Note that for any $\phi\in \mathbb W$, we have the solution $u(t,\cdot,\phi)\gg0$ for any $t>0$. We now prove the zero is a uniform weak repeller for $\mathbb W$ in the sense that there exists $\delta_0>0$ such that $\limsup_{t\to\infty}\|u(t,\cdot,\phi)\|_{X_1}\ge \delta_0$ for all $\phi\in \mathbb W$.
	Suppose, by contradiction, that $\limsup_{t\to\infty}\|u(t,\cdot,\phi_0)\|_{X_1}< \epsilon$ for some $\phi_0\in \mathbb W$. Then there exists $t_1>0$ such that $u_1(t,\cdot,\phi_0)<\epsilon$ and $u_2(t,\cdot,\phi_0)<\epsilon$ for any $t\ge t_1$ satisfying
	\begin{eqnarray}\label{PDEModeleq}
	&&\frac{\partial u}{\partial t}\ge d_1\Delta u-\alpha(x) u+\beta (x) v+u(m(x)-2\epsilon),\nonumber\\
	&&\frac{\partial v}{\partial t}\ge d_2 \Delta v+\alpha(x) u-\beta(x) v+v(m(x)-2\epsilon),\nonumber
	\end{eqnarray}
	Since $u(t_1,\cdot,\phi_0)$ is positive, there exists $a>0$ such that $u(t_1,\cdot,\phi_0)\ge a\phi^\epsilon$. Then the comparison principle implies that $u(t,\cdot,\phi_0)\ge ae^{\lambda^{\epsilon}(t-t_1)}\psi^\epsilon$ for any $t\ge t_1$. It then follows that $u(t,\cdot,\phi_0)$ is unbounded, which is impossible. 
	
	The above argument shows that $W^s(\{0\})\cap \mathbb W=\emptyset$ and $\{0\}$ is isolated in $X^+_1$, where $W^s(\{0\})$ is the stable set of $\{0\}$.  Define $p(\phi)=\min\limits_{1\le i\le2}\{\min\limits_{x\in\bar\Omega}\phi_i(x)\}$. It is easy to see that $p$ is a generalized distance function for the semiflow: $Q_t:X^+_1\to X^+_1$. 
	The dissipativity and precompactness of forward orbits for \eqref{PDEModel} imply that the the semi-dynamical system $Q_t(\phi):=u(t,\cdot,\phi)$  admits a compact global attractor on $\mathbb W$, and hence, it contains an equilibrium $(U^*,V^*)\in \mathbb W$. Moreover,  it 
	follows from \cite[Theorem 3]{SZ} that there exists an $\eta>0$ such that $\min\{p(\psi): \psi\in \omega(\phi)\}>\eta $ for any $\phi\in \mathbb W$. Therefore, 
	for any $\phi\in \mathbb W$, we have
	$$\liminf_{t\to\infty}u_i(t,x,\phi)\ge\eta, \quad\forall i=1,2.$$
\end{proof}
\section{System with small switching rates and positive $m(x)$} 
Throughout this section, we assume conditions in Proposition \ref{Cases}(1) hold and $bc\le1$. Roughly speaking, as long as positive $\beta$ and $\alpha$ are very small, the requirements in Proposition \ref{Cases}(1) would be valid. We consider the submodel
\begin{equation}\label{2PDEModel}
\aligned
&\frac{\partial u}{\partial t}=d_1\Delta u-\alpha u+\beta v+u(m(x)-u-bv)&&\mbox{in}\,(0,\infty)\times \Omega,\\
&\frac{\partial v}{\partial t}=d_2 \Delta v+\alpha u-\beta v+v(m(x)-cu-v)&&\mbox{in}\,(0,\infty)\times \Omega,\\
& \frac{\partial u}{\partial n}=\frac{\partial v}{\partial n}=0&&\mbox{on}\,(0,\infty)\times \partial\Omega,\\
&u(0,x)=\phi_1(x),\, v(0,x)=\phi_2(x)&&\mbox{in}\,\, \Omega. 
\endaligned
\end{equation}
 When $\alpha=\beta\equiv0$, this is the model studied in \cite{HN}.
 Since $(0,0)$ is unstable due to the fact $m>0$ on $\bar\Omega$ ,  the existence of the positive steady state follows immediately from Theorem \ref{UP}. 
By Proposition \ref{Cases}, we can show that every solution with positive initial data will eventually lie in the region $\displaystyle\left(\frac{\bar\beta}{b},\bar{m}\right)\times\left(\frac{\bar \alpha}{c},\bar{m}\right)$, where the system will be a competitive system.  Thus we can apply ideas based on the theory of positive  operators and monotone
semi-dynamical systems with respect to the competitive ordering.  

\begin{proposition}\label{AS}
	 If a positive steady state $(U,V)$ of \eqref{2PDEModel} exists, it must be asymptotically stable.
\end{proposition}
\begin{proof}
The essential idea is motivated by \cite{HN}.	Linearizing the steady state problem of \eqref{2PDEModel} at $(U,V)$,we have
	\begin{equation}
	\aligned
	&\lambda \Phi_1=d_1\Delta \Phi_1+(m(x)-2U-bV-\alpha)\Phi_1+(\beta-bU)\Phi_2,\quad &&\mbox{in}\,\,\Omega,\\
	&\lambda \Phi_2=d_2 \Delta \Phi_2+(\alpha-cV)\Phi_1+(m(x)-cU-2V-\beta)\Phi_2,\quad &&\mbox{in}\,\,\Omega,\\
	& \frac{\partial \Phi_1}{\partial n}=\frac{\partial \Phi_2}{\partial n}=0,\quad &&\mbox{on}\,\,\partial\Omega. 
	\endaligned
	\end{equation}
	By the Krein-Rutman theorem and the fact that $U(\cdot)>\frac{\bar\beta}{b}$ and $V(\cdot)>\frac{\bar\alpha}{c}$, the  eigenvalue problem admits a principal eigenvalue $\lambda_1$, with the corresponding eigenfunction satisfying $\Phi^*_1>0>\Phi^*_2$ on $\bar \Omega$.   
	By a straightforward calculation, using the equations satisfied
	by $U$ and $\Phi_1^*$ (multiply $U$-equation by $\Phi^*_1$ and $\Phi^*_1$-equation by $U$, and then do the subtraction) and the identity $(\Delta U)\Phi^*_1-U(\Delta\Phi^*_1)=\triangledown\cdot[(\triangledown U) \Phi^*_1-(\triangledown \Phi^*_1) U]$,  we obtain that
	\begin{equation}\label{u0}
	-\lambda_1\Phi^*_1 U=-d_1\triangledown\cdot\left(U^2\triangledown\frac{\Phi^*_1}{U}\right)  +U^2(\Phi^*_1+b\Phi^*_2)+\beta(V\Phi^*_1-U\Phi^*_2).
	\end{equation}
	Similarly, we have
	\begin{eqnarray*}
		&&-\lambda_1\Phi^*_2 V=-d_2\triangledown\cdot\left(V^2\triangledown\frac{\Phi^*_2}{V}\right)+V^2(c\Phi^*_1+\Phi^*_2)+\alpha(U\Phi^*_2-V\Phi^*_1).
	\end{eqnarray*}
	Multiplying both sides of \eqref{u0} by $\displaystyle\frac{\Phi^{*2}_1}{U^2}$ and integrating over $\Omega$, we see that
	\begin{equation}\label{phi1}
	-\lambda_1\int_{\Omega}\frac{\Phi_1^{*3}}{U}=2d_1\int_{\Omega}U\Phi^*_1\left|\triangledown\frac{\Phi^*_1}{U}\right|^2+\int_{\Omega}\Phi^{*2}_1(\Phi^*_1+b\Phi^*_2)+\int_{\Omega}\beta(V\Phi^*_1-U\Phi^*_2)\displaystyle\frac{\Phi^{*2}_1}{U^2}.
	\end{equation}
	Likewise, we get
	\begin{equation}\label{phi23}
	-\lambda_1\int_{\Omega}\frac{\Phi_2^{*3}}{V}=2d_2\int_{\Omega}V\Phi^*_2\left|\triangledown\frac{\Phi^*_2}{V}\right|^2+\int_{\Omega}\Phi^{*2}_2(c\Phi^*_1+\Phi^*_2)+\int_{\Omega}\alpha(U\Phi^*_2-V\Phi^*_1)\displaystyle\frac{\Phi^{*2}_2}{V^2}.
	\end{equation}
	Subtract \eqref{phi23} from the product of $c^3$ and \eqref{phi1}. Then together with the fact that $bc\le1$ and $\Phi^*_2<0$, we obtain 
	\begin{eqnarray}\label{phi}
	&&-\lambda_1\left(c^3\int_{\Omega}\frac{\Phi_1^{*3}}{U}-\int_{\Omega}\frac{\Phi_2^{*3}}{V}\right)\ge 2 c^3d_1\int_{\Omega}U\Phi^*_1\left|\triangledown\frac{\Phi^*_1}{U}\right|^2-2 d_2\int_{\Omega}V\Phi^*_2\left|\triangledown\frac{\Phi^*_2}{V}\right|^2\nonumber\\
	&&+\int_{\Omega}(c\Phi^{*}_1+\Phi^{*}_2)^2(c\Phi^*_1-\Phi^*_2)+\int_{\Omega}(V\Phi^*_1-U\Phi^*_2)\left(c^3\beta\displaystyle\frac{\Phi^{*2}_1}{U^2}+\alpha\displaystyle\frac{\Phi^{*2}_2}{V^2}\right).
	\end{eqnarray}
	It follows immediately from $\Phi^*_2<0$ and $V\Phi^*_1-U\Phi^*_2>0$ that the right hand side of \eqref{phi} is greater than zero, and hence, $\lambda_1<0$. Now it follows  immediately from \cite[Theorem 7.6.2]{H} that linearly stable ($\lambda_1<0$) implies asymptotically stable. 
\end{proof}
The following result is a direct consequence of Theorem \ref{UP}, Proposition \ref{AS} and monotone dynamical systems approach (see, e.g., \cite[Theorem 9.2]{Hess}).
\begin{theorem}
 If the conditions of Proposition \ref{Cases}(1) are satisfied and $bc \le 1$, then system \eqref{2PDEModel} admits a unique positive steady state $(U^*,V^*)$, which is globally asymptotically stable in $X^+_1\setminus\{0\}$.
	\end{theorem}

\section{Cooperative-cooperative-competition system}
In this section, we consider one species having two different kinds of movements that competes with another ecologically identical species having only one movement mode. 
Now consider the  system 
\begin{equation}\label{PDEModel2}
\aligned
&\frac{\partial u}{\partial t}=d_1\Delta u-\alpha u+\beta v+u(m(x)-u-v-w)&&\mbox{in}\,(0,\infty)\times \Omega,\\
&\frac{\partial v}{\partial t}=d_2 \Delta v+\alpha u-\beta v+v(m(x)-u-v-w)&&\mbox{in}\,(0,\infty)\times \Omega,\\
&\frac{\partial w}{\partial t}=d_3 \Delta w+w(m(x)-u-v-w)&&\mbox{in}\,(0,\infty)\times \Omega,\\
& \frac{\partial u}{\partial n}=\frac{\partial v}{\partial n}=\frac{\partial w}{\partial n}=0&&\mbox{on}\,(0,\infty)\times \partial\Omega,\\
&u(0,x)=\phi_1(x),\, v(0,x)=\phi_2(x), \, w(0,x)=\phi_3(x)&&\mbox{in}\,\, \Omega.
\endaligned
\end{equation}
Here $d_1<d_2$, $d_3$, $\alpha$ and $\beta$ are positive numbers. Throughout this section, we impose the following assumption.

\noindent (H) $m$ is non-constant, $\int_{\Omega}m\ge0$ and  $0<\max\limits_{\bar\Omega}{m(x)}<\alpha+\beta$.

By (H) and Proposition \ref{Cases},  one can show that the subsystem 
\begin{equation}\label{subPDEModel}
\aligned
&\frac{\partial u}{\partial t}=d_1\Delta u-\alpha u+\beta v+u(m(x)-u-v)&&\mbox{in}\,(0,\infty)\times \Omega,\\
&\frac{\partial v}{\partial t}=d_2 \Delta v+\alpha u-\beta v+v(m(x)-u-v)&&\mbox{in}\,(0,\infty)\times \Omega,\\
& \frac{\partial u}{\partial n}=\frac{\partial v}{\partial n}=0&&\mbox{on}\,(0,\infty)\times \partial\Omega,\\
&u(0,x)=\phi_1(x),\, v(0,x)=\phi_2(x)&&\mbox{in}\,\, \Omega. 
\endaligned
\end{equation}
is cooperative, irreducible and sub-homogeneous in a contracting rectangular region $[0,\beta]\times[0,\alpha]$ which attracts  all positive trajectories.  The approach of  monotone dynamical systems, along with  Proposition \ref{ST}, implies system \eqref{subPDEModel} admits a globally attractively positive steady state $(u^*,v^*)$.   (See \cite{CCY} for related results in the constant coefficient case.)

Because the dynamics of the first two components move the system \eqref{subPDEModel} into a region where they satisfy a cooperative system, we can treat the model \eqref{PDEModel2} as a monotone system with respect to the ordering $(u_1,v_1,w_1) \ge (u_2,v_2,w_2) \iff u_1 \ge u_2, v_1 \ge v_2, w_1\le w_2$.  Systems of ordinary differential equations with this type of order structure are treated in \cite{Smith1}; see also the discussion of alternate cones in \cite{Smith2}.  The ideas extend directly to reaction-diffusion systems via the maximum principle.

Also, the classic  result on logistic-type reaction-diffusion equations shows that 
\begin{equation}\label{subPDEModel2}
\aligned
&\frac{\partial w}{\partial t}=d_3 \Delta w+w(m(x)-w)&&\mbox{in}\,(0,\infty)\times \Omega,\\
&\frac{\partial w}{\partial n}=0&& \mbox{on}\,(0,\infty)\times \partial \Omega,\\
&w(0,x)=\phi(x) && \mbox{in}\, \Omega
\endaligned
\end{equation}
admits a globally attractively positive steady state $w^*(\cdot)$.  

The following observation is based on the strong maximum principle.
\begin{proposition}\label{SS}
Assume that (H) holds. Then the nontrivial nonnegative steady states of system \eqref{PDEModel2} are  $(u^*,v^*,0)$, $(0,0,w^*)$ plus any positive steady states that exist.
\end{proposition}

Next, we investigate the effects of diffusion rate $d_3$ on the local stability of $(u^*,v^*,0)$ and $(0,0,w^*)$; that is, we fix the other parameters and let $d_3$ vary. 
\begin{lemma}\label{uv}
	There exists  $d_c\in(d_1,\frac{\beta}{\alpha+\beta}d_1+\frac{\alpha}{\alpha+\beta}d_2)$, such that  $(u^*,v^*,0)$ is linearly unstable when $d_3<d_c$ and $(u^*,v^*,0)$ is linearly stable when $d_3>d_c$.
\end{lemma}
\begin{proof}
	To investigate the local stability of $(u^*,v^*,0)$, we consider the eigenvalue problem  
	\begin{equation}\label{EP8}
	\aligned
	&\lambda \phi_1=d_1\Delta \phi_1+(m(x)-2u^*-v^*-\alpha)\phi_1+(\beta-u^*)\phi_2-u^*\phi_3&&\mbox{in}\,\,\Omega,\\
	&\lambda \phi_2=d_2 \Delta \phi_2+(\alpha-v^*)\phi_1+(m(x)-u^*-2v^*-\beta)\phi_2-v^*\phi_3&&\mbox{in}\,\,\Omega,\\
	&\lambda\phi_3=d_3 \Delta \phi_3+(m(x)-u^*-v^*)\phi_3&&\mbox{in}\,\,\Omega,\\
	& \frac{\partial \phi_1}{\partial n}=\frac{\partial \phi_2}{\partial n}=\frac{\partial \phi_3}{\partial n}=0&&\mbox{on}\,\,\partial\Omega.
	\endaligned
	\end{equation}  
In view of Proposition \ref{Cases}(2), we have  $u^*<\beta$ and $v^*<\alpha$,  so this eigenvalue problem admits a principal eigenvalue which exactly is $\lambda(d_3, m-u^*-v^*)$ defined in Proposition \ref{PE}. Since $m$ is non-constant, it follows that $(u^*,v^*)$ is a non-constant steady state (that is, $u^*$ and $v^*$ are not both constant). Moreover, $m-u^*-v^*$ is also non-constant.
	 Otherwise, suppose $m-u^*-v^*\equiv K$. Then adding  the equations for the equilibria of \eqref{subPDEModel} together and integrating, we get $K\int_{\Omega}[u^*+v^*]=0$, and hence, $K=0$. It follows that
	 0 is the principal eigenvalue of \begin{equation*}
	 \aligned
	 &\lambda \phi_1=d_1\Delta \phi_1-\alpha \phi_1+\beta \phi_2&&\mbox{in}\,\,\Omega,\\
	 &\lambda \phi_1=d_2 \Delta \phi_2+\alpha \phi_1-\beta\phi_2&&\mbox{in}\,\,\Omega,\\
	 &\frac{\partial \phi_1}{\partial n}=\frac{\partial \phi_2}{\partial n}=0&&\mbox{on}\,\,\partial\Omega
	 \endaligned
	 \end{equation*}
	  associated with the positive eigenfunction $(u^*,v^*)$. Note that $(\beta,\alpha)$ is another positive eigenfunction associated with the principal eigenvalue $0$, and hence, $(u^*,v^*)^T\in Span\{(\beta,\alpha)^T\}$, which is impossible.

	 Observe that $(u^*,v^*,0)$ is independent of $d_3$. Then $\lambda(d_3,m-u^*-v^*)$ is continuous and strictly decreasing in $d_3$, that is, it changes sign at most once.  
	
\noindent Claim 1: $\lambda(d_3,m-u^*-v^*)>0$ when $d_3=d_1$.

	Suppose by contradiction, $\lambda(d_1,m-u^*-v^*)\le0$. Note that the non-constant steady state $(u^*,v^*)$ satisfies
	\begin{equation}\label{PDE4}
	\aligned
	&0=d_1\Delta u^*-\alpha u^*+\beta v^*+u^*(m(x)-u^*-v^*)&&\mbox{in}\,\,\Omega,\\
	&0=d_2 \Delta v^*+\alpha u^*-\beta v^*+v^*(m(x)-u^*-v^*)&&\mbox{in}\,\,\Omega,\\
	& \frac{\partial u^*}{\partial n}=\frac{\partial v^*}{\partial n}=0&&\mbox{on}\,\,\partial\Omega.
	\endaligned
	\end{equation}
	Multiplying the first and second equations by $\alpha u^*$ and $\beta v^*$, respectively, and then integrating over $\Omega$ and adding together,  we see that
	\begin{eqnarray}\label{Ineq2}
	&&\alpha\left[-d_1\int_\Omega|\triangledown u^*|^2+\int_{\Omega}(m-u^*-v^*)u^{*2}\right]\nonumber\\
	&&+\beta\left[-d_2\int_\Omega|\triangledown v^*|^2+\int_{\Omega}(m-u^*-v^*)v^{*2}\right]\\
	&&=\int_{\Omega}(\alpha u^*-\beta v^*)^2\ge0.\nonumber
	\end{eqnarray}
	Since $d_1<d_2$, $\lambda(d_2,m^*-u^*-v^*)<\lambda(d_1,m^*-u^*-v^*)\le0$,  and it follows from the variational formula for the principal eigenvalue $\lambda(d,m)$ that  $-d_1\int_\Omega|\triangledown u^*|^2+\int_{\Omega}(m-u^*-v^*)u^{*2}\le0$ and $-d_2\int_\Omega|\triangledown v^*|^2+\int_{\Omega}(m-u^*-v^*)v^{*2}<0$, in contradiction to \eqref{Ineq2}.
	
	Claim 2: $\lambda(d_3,m^*-u^*-v^*)<0$ when $d_3=\frac{\beta}{\alpha+\beta}d_1+\frac{\alpha}{\alpha+\beta}d_2$.
	
	By way of contradiction, assume that $\lambda(d^0_3,m-u^*-v^*)\ge0$ where $d^0_3=\frac{\beta}{\alpha+\beta}d_1+\frac{\alpha}{\alpha+\beta}d_2$. Let $\phi^*$ be the positive eigenfunction associated with $\lambda(d^0_3,m-u^*-v^*)$; clearly, it is non-constant.
	
	Let $$L\left(\begin{array}{c}
	\phi_1\\\phi_2
	\end{array}\right)=\left(\begin{array}{c}
	\alpha d_1\Delta \phi_1+[(m-u^*-v^*)\alpha-\alpha^2]\phi_1+ \alpha\beta \phi_2\\
	\beta d_2\Delta \phi_2+\alpha\beta \phi_1+[(m-u^*-v^*)\beta-\beta^2]\phi_2
	\end{array}\right).$$ Then $L$ is a self-adjoint operator. The principal eigenvalue of $L$ is $0$ with $(u^*,v^*)$ being the associated eigenfunction,  and we have the variational formula for the principal eigenvalue of $L$  
	\begin{eqnarray*}
		0=\lambda(L)=\sup_{(\phi_1,\phi_2)\in H^1(\Omega,\R^2)\setminus\{0\}} \left\{\frac{\alpha\left[-d_1\int_\Omega|\triangledown \phi_1|^2+\int_{\Omega}(m-u^*-v^*)\phi^2_1\right]}{\int_\Omega{(\phi^2_1+\phi_2^2)}}\right.
		\\
		\left. +\frac{\beta\left[-d_2\int_\Omega|\triangledown \phi_2|^2+\int_{\Omega}(m-u^*-v^*)\phi_2^2\right]}{\int_\Omega{(\phi^2_1+\phi_2^2)}}-\frac{\int_{\Omega}(\alpha \phi_1-\beta \phi_2)^2}{\int_\Omega{(\phi^2_1+\phi_2^2)}}
		\right\}.
	\end{eqnarray*}
Choose test functions $\phi_1=\frac{\phi^*}{\alpha}$ and $\phi_2=\frac{\phi^*}{\beta}$. It then follows that \begin{eqnarray}\label{InEq2}
	&&\frac{-d_1\int_\Omega|\triangledown \phi^*|^2+\int_{\Omega}(m-u^*-v^*)\phi^{*2}}{\alpha}\nonumber\\
	&&+\frac{-d_2\int_\Omega|\triangledown \phi^*|^2+\int_{\Omega}(m-u^*-v^*)\phi^{*2}}{\beta}<0.
	\end{eqnarray}
	The above strict inequality is due to the fact that $\alpha u^*-\beta v^*$ is not identically to zero and hence  $(\phi_1,\phi_2)\not=(u^*,v^*)$. (If $\alpha u^*-\beta v^*\equiv 0$, it then follows that $\lambda(d_1,m-u^*-v^*)=\lambda(d_2,m-u^*-v^*)=0$ with $d_1<d_2$, a contradiction.)
	
	Since we assumed $-d^0_3\int_\Omega|\triangledown \phi^*|^2+\int_{\Omega}(m-u^*-v^*)\phi^{*2}\ge0$, we have  \begin{eqnarray}\label{InEq3}
	&&\left(\frac{d^0_3-d_1}{\alpha}+\frac{d^0_3-d_2}{\beta}\right)\int_{\Omega}|\triangledown \phi^*|^2\le\frac{-d_1\int_\Omega|\triangledown \phi^*|^2+\int_{\Omega}(m-u^*-v^*)\phi^{*2}}{\alpha}\nonumber\\
	&&+\frac{-d_2\int_\Omega|\triangledown \phi^*|^2+\int_{\Omega}(m-u^*-v^*)\phi^{*2}}{\beta}<0.
	\end{eqnarray}
	This implies $\displaystyle d^0_3< \frac{\beta}{\alpha+\beta}d_1+\frac{\alpha}{\alpha+\beta}d_2$, a contradiction.
	
	From the above discussion, we see that given $d_1,d_2,\alpha,\beta$, there exists a unique $d_{c}\in(d_1,\frac{\beta}{\alpha+\beta}d_1+\frac{\alpha}{\alpha+\beta}d_2)$, such that $\lambda(d_c, m-u^*-v^*)=0$. When $d_3<d_{c}$, $(u^*,v^*,0)$ is linearly unstable, while $d_3>d_{c}$, $(u^*,v^*,0)$ is linearly stable.
\end{proof}

Likewise, we check the local stability of $(0,0,w^*)$. The associated eigenvalue problem is
\begin{equation}\label{EP6}
\aligned
&\lambda \phi_1=d_1\Delta \phi_1+(m(x)-w^*-\alpha)\phi_1+\beta\phi_2&&\mbox{in}\,\,\Omega,\\
&\lambda \phi_2=d_2 \Delta \phi_2+\alpha\phi_1+(m(x)-\beta-w^*)\phi_2&&\mbox{in}\,\,\Omega,\\
&\lambda\phi_3=d_3 \Delta \phi_3-w^*\phi_1-w^*\phi_2+(m(x)-2w^*)\phi_3&&\mbox{in}\,\,\Omega,\\
& \frac{\partial \phi_1}{\partial n}=\frac{\partial \phi_2}{\partial n}=\frac{\partial \phi_3}{\partial n}=0&&\mbox{on}\,\,\partial\Omega.
\endaligned
\end{equation}  
The principal eigenvalue $\lambda_2$ of \eqref{EP6} is determined by the sub-eigenvalue problem 
\begin{equation}\label{EP7}
\aligned
&\lambda \phi_1=d_1\Delta \phi_1+(m(x)-w^*-\alpha)\phi_1+\beta\phi_2&&\mbox{in}\,\,\Omega,\\
&\lambda \phi_2=d_2 \Delta \phi_2+\alpha\phi_1+(m(x)-\beta-w^*)\phi_2&&\mbox{in}\,\,\Omega,\\
& \frac{\partial \phi_1}{\partial n}=\frac{\partial \phi_2}{\partial n}=0&&\mbox{on}\,\,\partial\Omega.
\endaligned
\end{equation}  
This eigenvalue problem is equivalent to the weighted eigenvalue problem
\begin{equation}\label{EP7A}
\aligned
&\lambda \alpha \phi_1=d_1\alpha \Delta \phi_1+(m(x)-w^*-\alpha)\alpha\phi_1+\alpha \beta\phi_2&&\mbox{in}\,\,\Omega,\\
&\lambda \beta \phi_2=d_2 \beta \Delta \phi_2+\alpha\beta \phi_1+(m(x)-\beta-w^*)\beta \phi_2&&\mbox{in}\,\,\Omega,\\
& \frac{\partial \phi_1}{\partial n}=\frac{\partial \phi_2}{\partial n}=0&&\mbox{on}\,\,\partial\Omega.
\endaligned
\end{equation}  
The eigenvalue problem \eqref{EP7A} is self-adjoint, so it admits a variational characterization, from which it is easy to see that $\lambda_2$ depends continuously on $w^*$. General properties of solutions to diffusive logistic equations imply that $w^*$ depends smoothly on $d_3>0$  (See for example \cite{CCbook}). This implies $\lambda_2$ depends continuously on $d_3>0$. Now we have the following result.

\begin{lemma}\label{w}
	Assume that (H) holds. If $d_3\le d_1$, then $\lambda_2(d_3)<0$, that is, $(0,0,w^*)$ is linearly stable. If $d_3\ge\frac{\beta}{\alpha+\beta}d_1+\frac{\alpha}{\alpha+\beta}d_2 $, then $\lambda_2(d_3)>0$, that is,  $(0,0,w^*)$ is linearly unstable. Moreover, there exists $d_0\in(d_1,\frac{\beta}{\alpha+\beta}d_1+\frac{\alpha}{\alpha+\beta}d_2)$, such that $\lambda_2(d_0)=0$
\end{lemma}
\begin{proof}
	Let $(\phi_1,\phi_2)$ be the positive eigenfunction associated with $\lambda_2$. Then  
	multiplying the first and second equations of \eqref{EP7} by $\alpha \phi_1$ and $\beta \phi_2$, respectively, and then integrating over $\Omega$,  we see that
	\begin{eqnarray}\label{Ineq3}
	&&\alpha\left[-d_1\int_\Omega|\triangledown \phi_1|^2+\int_{\Omega}(m-w^*)\phi_1^{2}\right]+\beta\left[-d_2\int_\Omega|\triangledown \phi_2|^2+\int_{\Omega}(m-w^*)\phi_2^{2}\right]\nonumber\\
	&&=\int_{\Omega}(\alpha \phi_1-\beta \phi_2)^2+\lambda_2\int_\Omega{\alpha\phi^2_1+\beta\phi_2^2}.
	\end{eqnarray}
	
	In the case where $d_3\le d_1$, we claim that $\lambda_2<0$. Otherwise, there exists some $\tilde{d}_3 \le d_1$ such that $\lambda_2\ge0$. It then follows from $\eqref{Ineq3}$ that  for $w^*=w^*(\tilde{d}_3)$, $$\alpha\left[-d_1\int_\Omega|\triangledown \phi_1|^2+\int_{\Omega}(m-w^*)\phi_1^{2}\right]
	+\beta\left[-d_2\int_\Omega|\triangledown \phi_2|^2+\int_{\Omega}(m-w^*)\phi_2^{2}\right]\ge0.$$ However, the fact $\lambda(\tilde{d}_3,m-w^*)=0$ and $\tilde{d}_3\le d_1<d_2$ implies that $\lambda(d_2,m-w^*)<\lambda(d_1,m-w^*)\le0$, and hence $-d_1\int_\Omega|\triangledown \phi_1|^2+\int_{\Omega}(m-w^*)\phi_1^{2}\le0$ and $-d_2\int_\Omega|\triangledown \phi_2|^2+\int_{\Omega}(m-w^*)\phi_2^{2}<0$, which is a contradiction.
	
 Suppose there exists some $d_3\ge\frac{\beta}{\alpha+\beta}d_1+\frac{\alpha}{\alpha+\beta}d_2$ such that $\lambda_2\le0$.   Adapting the previous analysis to \eqref{EP7} by writing down the variational formula arising from \eqref{EP7A}, we have
	\begin{eqnarray*}
		0\ge\lambda_2=\sup_{(\phi_1,\phi_2)\in H^1(\Omega,\R^2)\setminus\{0\}} \left\{\frac{\alpha\left[-d_1\int_\Omega|\triangledown \phi_1|^2+\int_{\Omega}(m-w^*)\phi^2_1\right]}{\int_\Omega{(\alpha\phi^2_1+\beta\phi_2^2)}}\right.
		\\
		\left. +\frac{\beta\left[-d_2\int_\Omega|\triangledown \phi_2|^2+\int_{\Omega}(m-w^*)\phi_2^2\right]}{\int_\Omega{(\alpha\phi^2_1+\beta\phi_2^2})}-\frac{\int_{\Omega}(\alpha \phi_1-\beta \phi_2)^2}{\int_\Omega{(\alpha\phi^2_1+\beta\phi_2^2})}
		\right\}.
	\end{eqnarray*}
	
	Choose test functions $\phi_1=\frac{w^*}{\alpha}$ and $\phi_2=\frac{w^*}{\beta}$. If  $(\frac{w^*}{\alpha},\frac{w^*}{\beta})$ were an eigenfunction for \eqref{EP7A},  we could substitute into  the two equations in \eqref{EP7A} and subtract to see that $w^*$ would satisfy $(d_2-d_1)\Delta w^*=0$ with Neumann boundary conditions and hence would be constant, but $w^*$ cannot be constant,  so $(\frac{w^*}{\alpha},\frac{w^*}{\beta})$ cannot be an eigenfunction for \eqref{EP7A}. Thus it  follows that

	\begin{eqnarray*}\label{InEq5}
	&&\frac{-d_1\int_\Omega|\triangledown w^*|^2+\int_{\Omega}(m-w^*)w^{*2}}{\alpha}\nonumber+\frac{-d_2\int_\Omega|\triangledown w^*|^2+\int_{\Omega}(m-w^*)w^{*2}}{\beta}<0.
	\end{eqnarray*}
	Since $-d_3\int_\Omega|\triangledown w^*|^2+\int_{\Omega}(m-w^*)w^{*2}=0$, we have  \begin{eqnarray}\label{InEq6}
	&&\left(\frac{d_3-d_1}{\alpha}+\frac{d_3-d_2}{\beta}\right)\int_{\Omega}|\triangledown w^*|^2\le\frac{-d_1\int_\Omega|\triangledown w^*|^2+\int_{\Omega}(m-w^*)w^{*2}}{\alpha}\nonumber\\
	&&+\frac{-d_2\int_\Omega|\triangledown w^*|^2+\int_{\Omega}(m-w^*)\phi^{*2}}{\beta}<0.
	\end{eqnarray}
	This implies $\displaystyle d_3< \frac{\beta}{\alpha+\beta}d_1+\frac{\alpha}{\alpha+\beta}d_2$, a contradiction. 
	
	Since $\lambda_2$ depends continuously on $d_0$, there exists some $d_0\in(d_1,\frac{\beta}{\alpha+\beta}d_1+\frac{\alpha}{\alpha+\beta}d_2)$ such that $\lambda_2(d_0)=0$.
\end{proof}
\begin{remark}
	Here we are unable to show that there exists a unique $d_3>0$ such that $\lambda_2(d_3)=0$.
\end{remark}
Next we make an  observation on the nonexistence of positive steady states of \eqref{PDEModel2}.
\begin{lemma}\label{no}There exists a sufficiently small $\epsilon>0$ such that system \eqref{PDEModel2} admits no positive steady state (that is, no coexistence state)  when $d_3\in(0,d_1+\epsilon)\cup(\frac{\beta}{\alpha+\beta}d_1+\frac{\alpha}{\alpha+\beta}d_2-\epsilon,\infty)$. 
\end{lemma}
\begin{proof}First, we prove that when $d_3\le d_1$ and $d_3\ge\frac{\beta}{\alpha+\beta}d_1+\frac{\alpha}{\alpha+\beta}d_2$, there is no positive steady state. The essential idea is similar to those in Lemmas 3.2 and 3.3. 
	Suppose that, by contradiction, $(u_0,v_0,w_0)$ is a positive steady state of \eqref{PDEModel2}. Then
	\begin{equation}\label{PDE3}
	\aligned
	&0=d_1\Delta u_0-\alpha u_0+\beta v_0+u_0(m(x)-u_0-v_0-w_0)&&\mbox{in}\,\,\Omega,\\
	&0=d_2 \Delta v_0+\alpha u_0-\beta v_0+v_0(m(x)-u_0-v_0-w_0)&&\mbox{in}\,\,\Omega,\\
	&0=d_3 \Delta w_0+w_0(m(x)-u_0-v_0-w_0)&&\mbox{in}\,\,\Omega,\\
	& \frac{\partial u_0}{\partial n}=\frac{\partial v_0}{\partial n}=\frac{\partial w_0}{\partial n}=0&&\mbox{on}\,\,\partial\Omega.
	\endaligned
	\end{equation}
	Note that $m-u_0-w_0-w_0$ is non-constant. Otherwise, as before, we can show $u_0$ and $v_0$ have to be constant. This implies $m$ is constant, impossible.
	
	Consider the case where $d_3\le d_1<d_2$. Multiplying the first and second equations by $\alpha u_0$ and $\beta v_0$, respectively, and then integrating over $\Omega$,  we see that
	\begin{eqnarray}\label{Ineq}
	&&\alpha\left[-d_1\int_\Omega|\triangledown u_0|^2+\int_{\Omega}(m-u_0-v_0-w_0)u^2_0\right]\nonumber\\
	&&+\beta\left[-d_2\int_\Omega|\triangledown v_0|^2+\int_{\Omega}(m-u_0-v_0-w_0)v^2_0\right]\\
	&&=\int_{\Omega}(\alpha u_0-\beta v_0)^2\ge0.\nonumber
	\end{eqnarray}
 Since $0=\lambda(d_3,m-u_0-v_0-w_0)\ge\lambda(d_1,m-u_0-v_0-w_0)>\lambda(d_2,m-u_0-v_0-w_0)$,
	the variational formula of the principal eigenvalue implies that $-d_1\int_\Omega|\triangledown u_0|^2+\int_{\Omega}(m-u_0-v_0-w_0)u^2_0\le0$ and $-d_2\int_\Omega|\triangledown v_0|^2+\int_{\Omega}(m-u_0-v_0-w_0)v^2_0<0$, a contradiction.
	
	In the case where $d_3\ge\frac{\beta}{\alpha+\beta}d_1+\frac{\alpha}{\alpha+\beta}d_2$,  let $$L\left(\begin{array}{c}
	\phi_1\\\phi_2
	\end{array}\right)=\left(\begin{array}{c}
	\alpha d_1\Delta \phi_1+[(m-u_0-v_0-w_0)\alpha-\alpha^2]\phi_1+ \alpha\beta \phi_2\\
	\beta d_2\Delta \phi_2+\alpha\beta \phi_1+[(m-u_0-v_0-w_0)\beta-\beta^2]\phi_2
	\end{array}\right).$$ Then $L$ is a self-adjoint operator. The principal eigenvalue of $L$ is 0, and we have the variational formula for the principal eigenvalue of $L$  
	\begin{eqnarray*}
		\lambda(L)=\sup_{(\phi_1,\phi_2)H^1(\Omega,\R^2)\setminus\{0\}} \left\{\frac{\alpha\left[-d_1\int_\Omega|\triangledown \phi_1|^2+\int_{\Omega}(m-u_0-v_0-w_0)\phi^2_1\right]}{\int_\Omega{\phi^2_1+\phi_2^2}}\right.
		\\
		\left. +\frac{\beta\left[-d_2\int_\Omega|\triangledown \phi_2|^2+\int_{\Omega}(m-u_0-v_0-w_0)\phi_2^2\right]}{\int_\Omega{\phi^2_1+\phi_2^2}}-\frac{\int_{\Omega}(\alpha \phi_1-\beta \phi_2)^2}{\int_\Omega{\phi^2_1+\phi_2^2}}
		\right\}
	\end{eqnarray*}
	Now let $\phi_1=\frac{w_0}{\alpha}$ and $\phi_2=\frac{w_0}{\beta}$. It follows that
	\begin{equation*}
	\frac{-d_1\int_\Omega|\triangledown w_0|^2+\int_{\Omega}(m\!-\!u_0-v_0-w_0)w^2_0}{\alpha}+\frac{-d_2\int_\Omega|\triangledown w_0|^2+\int_{\Omega}(m-u_0-v_0-w_0)w^2_0}{\beta}
	\end{equation*}
	is negative.
	Since $-d_3\int_\Omega|\triangledown w_0|^2+\int_{\Omega}(m-u_0-v_0-w_0)w^2_0=0$, it follows that
	$\left(\frac{d_3-d_1}{\alpha}+\frac{d_3-d_2}{\beta}\right)\int_{\Omega}|\triangledown w_0|^2<0$. This yields that $\displaystyle d_3< \frac{\beta}{\alpha+\beta}d_1+\frac{\alpha}{\alpha+\beta}d_2$, a contradiction.
	
	Motivated by \cite{CCL07}, now we suppose that when $d_3\to d^+_1$, there exists a sequence of positive steady states, denoted by  $(u^{d_3},v^{d_3},w^{d_3})$. By standard elliptic estimates, passing to a subsequence if necessary, we may assume that $(u^{d_3},v^{d_3},w^{d_3})\to(u_0,v_0,w_0)$ in $C^2(\bar\Omega)$ as $d_3\to d^+_1$ satisfying
	 \begin{equation}\label{eq5}
	 \aligned
	 &0=d_1\Delta u_0-\alpha u_0+\beta v_0+u_0(m(x)-u_0-v_0-w_0)&&\mbox{in}\,\,\Omega,\\
	 &0=d_2 \Delta v_0+\alpha u_0-\beta v_0+v_0(m(x)-u_0-v_0-w_0)&&\mbox{in}\,\,\Omega,\\
	 &0=d_1 \Delta w_0+w_0(m(x)-u_0-v_0-w_0)&&\mbox{in}\,\,\Omega,\\
	 & \frac{\partial u_0}{\partial n}=\frac{\partial v_0}{\partial n}=\frac{\partial w_0}{\partial n}=0&&\mbox{on}\,\,\partial\Omega.
	 \endaligned
	 \end{equation}
	It follows immediately from the previous analysis that when $d_3=d_1$ the non-negative steady state $(u_0,v_0,w_0)$ cannot be component-wise positive. 
	
	Suppose that $(u_0,v_0,w_0)\equiv(0,0,0)$ and let $$(\hat {u}^{d_3},\hat {v}^{d_3},\hat {w}^{d_3})=(\frac{u^{d_3}}{\|u^{d_3}\|_{L^\infty}+\|v^{d_3}\|_{L^\infty}},\frac{v^{d_3}}{\|u^{d_3}\|_{L^\infty}+\|v^{d_3}\|_{L^\infty}},\frac{w^{d_3}}{\|w^{d_3}\|_{L^\infty}}).$$ Divide the first two equations and the third equation of \eqref{eq5} by $\|u^{d_3}\|_{L^\infty}+\|v^{d_3}\|_{L^\infty}$ and $\|w^{d_3}\|_{L^\infty}$, respectively. Then using the elliptic estimates again, we may assume that 
 $(\hat u^{d_3}, \hat v^{d_3},\hat w^{d_3})\to(\hat u_0,\hat v_0,\hat w_0)$ in $C^2(\bar\Omega)$ as $d_3\to d^+_1$ satisfying
 \begin{equation}\label{eq6}
 \aligned
 &0=d_1\Delta \hat u_0+\hat u_0(m(x)-\alpha)+\beta\hat v_0&&\mbox{in}\,\,\Omega,\\
 &0=d_2 \Delta \hat v_0+\alpha \hat u_0+\hat v_0(m(x)-\beta)&&\mbox{in}\,\,\Omega,\\
 &0=d_1 \Delta \hat w_0+\hat w_0m(x)&&\mbox{in}\,\,\Omega,\\
 & \frac{\partial \hat u_0}{\partial n}=\frac{\partial \hat v_0}{\partial n}=\frac{\partial\hat w_0}{\partial n}=0&&\mbox{on}\,\,\partial\Omega.
 \endaligned
 \end{equation}
 The third equation in \eqref{eq6} implies that either $\int m<0$ or $\hat w_0\equiv0$, which contradicts (H) or $\|\hat w_0\|_{L^{\infty}}=1$. Thus $(u_0,v_0,w_0)$ cannot   be compnentwise-positive or have all components zero. 
 Now if non-zero $(u_0,v_0,w_0)$ has at least one component that is identically to zero, then $(u_0,v_0,w_0)$ must be either $(u^*,v^*,0)$ or $(0,0,w^*_{d_1})$ in view of Proposition \ref{SS}. 
 
 If the former case occurs,  we have $0=d_1 \Delta \hat w_0+\hat w_0(m(x)-u^*-v^*)$ with zero Neumann boundary condition, where $\hat w_0$ is the limit of $\hat w^{d_3}=\frac{w^{d_3}}{\|w^{d_3}\|_{L^\infty}}$ as $d_3\to d^+_1$. Since $\|\hat w_0\|_{L^{\infty}}=1$, we see from the strong  maximum principle (looking at $w_t=d_1 \Delta \hat w_0+\hat w_0(m(x)-u^*-v^*)$)  that $\hat w_0>0$ in $\bar \Omega$, and hence, $\lambda(d_1,m-u^*-v^*)=0$, which contradicts Lemma \ref{uv}.
 
 Likewise, if the latter case occurs, we have   \begin{equation}\aligned\label{eq7}
 &0=d_1\Delta\hat u_0-\alpha \hat u_0+\beta \hat v_0+\hat u_0(m(x)-w^*_{d_1})&&\mbox{in}\,\,\Omega,\\
 &0=d_2 \Delta \hat v_0+\alpha \hat u_0-\beta \hat v_0+\hat v_0(m(x)-w^*_{d_1})&&\mbox{in}\,\,\Omega,\\
 &\frac{\partial \hat u_0}{\partial n}=\frac{\partial \hat v_0}{\partial n}=0&&\mbox{on}\,\,\partial\Omega.
 \endaligned
 \end{equation} Now $\|\hat u_0\|_{L^{\infty}}+\|\hat v_0\|_{L^{\infty}}=1$ and $\hat u_0\ge0, \hat v_0\ge0$. Then $\hat u_0>0$ and $\hat v_0>0$ on $\bar{\Omega}$ due to the  maximum principle. This implies that $\lambda_2=0$ in \eqref{EP6} if $d_3=d_1$, which contradicts  Lemma \ref{w}. 
 
 We can use a similar indirect argument to prove that when $d_3\to d^-$ where $d=\frac{\beta}{\alpha+\beta}d_1+\frac{\alpha}{\alpha+\beta}d_2$, there is no positive steady state. For simplicity, we use the exact notation as before. Following the same process,  we show that the non-zero limiting steady state $(u_0,v_0,w_0)$ must be either $(u^*,v^*,0)$ or $(0,0,w^*_{d})$.
 If the former case happens, it gives $\lambda(d,m-u^*-v^*)=0<\lambda(d_c,m-u^*-v^*)=0$, a contradiction. If the latter case happens, it implies $\lambda_2=0$ when $d_3=d$, also a contradiction.
 
 Based on the above discussion, the result  follows.
 \end{proof}

  We can combine the results on the stability or instability of semi-trivial steady states with monotone dynamical systems theory to obtain some results on  the dynamics of \eqref{PDEModel2}.
Let $X_1=C(\bar\Omega,\R^2)$, $X_2=C(\bar\Omega,\R)$, $X^+_1=C(\bar\Omega,\R^2_+)$ and $X^+_2=C(\bar\Omega,\R_+)$.  As noted previously, \eqref{PDEModel2} generates a monotone semi-flow on $X_1 \times X_2$ with respect to the cooperative cooperative-competitive ordering.
 \begin{theorem}
 	Assume that (H) holds. Then there exist $d_1<C_1\le C_2<\frac{\beta}{\alpha+\beta}d_1+\frac{\alpha}{\alpha+\beta}d_2$ such that the following statements are valid for system \eqref{PDEModel2}.
 	\begin{enumerate}
 	\item[(i)] $(0,0,w^*)$ is globally asymptotically stable in $X^+_1\times (X^+_2\setminus\{0\})$ when $d_3\in(0,C_1)$.
 	\item[(ii)] $(u^*,v^*,0)$ is globally asymptotically stable in $(X^+_1\setminus\{0\})\times X^+_2$ when $d_3\in(C_2,\infty)$.
 	
	 	\end{enumerate}
  \end{theorem}
\begin{proof}[Sketch of proof:] 
We utilize the theory developed in \cite{Hsu} for abstract competitive systems (see also \cite{Hess2})
	to prove the global stability of one of the boundary steady states. Set $X=X_1\times X_2$, $K =X^+_1\times(-X^+_2)$ and $\text{Int} K =\text{Int}X^+_1\times(-\text{Int} X^+_2)$.  Then $K$ generates the partial order relations $\le_K$, $<_K$, $\ll_K$ on $X$. To prove statement (i) or (ii), we might set $E_0=(0,0)$, $E_1=(0,w^*)$, $E_2=(\hat{u},0)$ with $\hat{u}=(u^*,v^*)$.
	
	Clearly, (H1)--(H4) in \cite{Hsu} are valid. See also \cite{LM}.
	 Lemmas \ref{uv}-\ref{w} and Lemma \ref{no}, together with \cite[Theorem B]{Hsu}, implies statement (i) or (ii) is valid when $d_3\in(0,d_1+\epsilon)$ or $d_3\in(\frac{\beta}{\alpha+\beta}d_1+\frac{\alpha}{\alpha+\beta}d_2-\epsilon,\infty)$.
	 Now define 
	$$C_1:=\sup\{d: \text{there is no-coexistence steady state for}\, d_3\in(0,d)\},$$
	and 
	$$C_2:=\inf\{d: \text{there is no-coexistence steady state for}\, d_3\in(d,\infty)\}.$$
	Then it easily follows that $d_1<C_1\le d^0_3\le C_2< \frac{\beta}{\alpha+\beta}d_1+\frac{\alpha}{\alpha+\beta}d_2$. 
\end{proof}
Remark:  We expect that there are conditions under which the system \eqref{PDEModel2} has a coexistence state but we will not pursue that point here.

 \section{Effects of switching rates on the dynamics}
 Throughout this section, we assume that hypothesis (H) holds, so the results of Section 3 apply.  When $d_3\le d_1$ or $d_3\ge d_2$, the species having slower diffusion also wins the competition. In order to study the effects of switching rate on the competition we only focus on the case when $d_1<d_3<d_2$.
 
 \begin{lemma}\label{beta}Assume that $d_1<d_3<d_2$ and $\max_{x\in\bar\Omega}m(x)\le \alpha$.  Then the following statements are valid.
 	\begin{enumerate}
 		\item [(i)]There exists a unique $\beta_c\in(0,\frac{d_2-d_3}{d_3-d_1}\alpha)$, such that $(0,0,w^*)$ is linearly stable when $\beta\in(0,\beta_c)$; linearly unstable when $\beta\in(\beta_c,\infty)$.
 		\item[(ii)] $(u^*,v^*,0)$ is linearly unstable if $\beta$ is small enough; linearly stable if $\beta\in[\frac{d_2-d_3}{d_3-d_1}\alpha,\infty)$.
 		\item[(iii)]There exists small $\epsilon>0$, such that system \eqref{PDEModel2} admits no positive steady state when $\beta\in(0,\epsilon)\cup (\frac{d_2-d_3}{d_3-d_1}\alpha-\epsilon,\infty)$. 
 	\end{enumerate} 
 	\end{lemma} 
 \begin{proof}
 	 	For statement (i), it suffices to check the principal eigenvalue $\lambda_2$ of \eqref{EP6} (equivalently \eqref{EP7} or \eqref{EP7A}) in terms of $\beta$. We prove that $\lambda_2$ is continuously differentiable on $\beta>0$ by the implicit function theorem. 
 	Let $E=C^{2+\alpha}(\bar{\Omega},\R)\times C^{2+\alpha}(\bar{\Omega},\R)\times \R$ and $F=C^{\alpha}(\bar{\Omega},\R)\times C^{\alpha}(\bar{\Omega},\R)\times \R,$ $0<\alpha<1$, and consider a mapping $\Phi:E\times (0,\infty)\to F$ given by
 	$$\Phi(v_1,v_2,s,\beta)=\left(\begin{array}{lll}
 	d_1\Delta v_1+(m-w^*-\alpha) v_1+\beta v_2-sv_1\\ d_2\Delta v_2+(m-w^*-\beta) v_2+\alpha v_1-sv_2\\\int_\Omega{(v^2_1+v_2^2)}-1
 	\end{array}\right)$$
 Note that $\Phi$ is a continuous map and that the linearization of $\Phi$ with respect to $E$ at $(v_1,v_2,s,\beta)$, denoted $D_1\Phi(v_1,v_2,s,\beta): E\to F$, is given by
 	$$[D_1\Phi(v_1,v_2,s,\beta)](w_1,w_2,t)=\left(\begin{array}{lll}
 	d_1\Delta w_1+(m-w^*-\alpha) w_1+\beta w_2-sw_1-tv_1\\ d_2\Delta w_2+(m-w^*-\beta) w_2+\alpha w_1-sw_2-tv_2\\2\int_\Omega{(v_1w_1+v_2w_2)}
 	\end{array}\right).$$
 	Let $(v_{10},v_{20})=(\phi_1(\beta_0),\phi_2(\beta_0))$ and $s_0=\lambda_2(\beta_0)$. Here $(\phi_1(\beta_0),\phi_2(\beta_0))$ is the positive eigenfunction corresponding to $\lambda_2(\beta_0)$ with $\int_\Omega{[\phi_1(\beta_0)]^2+[\phi_2(\beta_0)]^2}=1$. Our  next goal is to show that $D_1\Phi(v_1,v_2,s,\beta)$ is a bijection. 
 	
 	Suppose then that $D_1\Phi(v_{10},v_{20},\lambda_2(\beta_0),\beta_0)(w_1,w_2,t)=(0,0,0).$ Then $d_1\Delta w_1+(m-w^*-\alpha-\lambda_2(\beta_0))w_1+\beta_0w_2=tv_{10}$, $d_2\Delta w_2+(m-w^*-\beta_0-\lambda_2(\beta_0))w_2+\alpha w_1=tv_{20}$ with zero Neumann boundary condition, and $\int_{\Omega}w_1v_{10}+w_2v_{20}=0$.
 	
 	Direct calculations similar to those in 
 	Proposition \ref{AS} indicate that
 	\begin{eqnarray}\label{v}
 &&	d_1\triangledown\cdot\left(v_{10}\triangledown w_{10}-w_{10}\triangledown v_{10}\right)+\beta_0(w_2v_{10}-v_{20}w_1)=tv^2_{10}\nonumber\\
 &&	d_2\triangledown\cdot\left(v_{20}\triangledown w_{20}-w_{20}\triangledown v_{20}\right)+\alpha(w_1v_{20}-v_{10}w_2)=tv^2_{20}\\
 &&	\frac{\partial w_1}{\partial n}=	\frac{\partial w_2}{\partial n}=0. \nonumber
 	\end{eqnarray}
  Multiply the equations of \eqref{v} by $\alpha$ and  $\beta_0$, respectively, then integrate over $\Omega$, and lastly add together. It then follows that  $0=t[\int_\Omega{\alpha v^2_{10}+\beta_0v_{20}^2}]$, and hence, $t=0$. Since $\lambda_2(\beta_0)$ is the principal eigenvalue of \eqref{EP7}, we see from the algebraic simplicity of $\lambda_2(\beta_0)$ that $(w_{10},w_{20})\in$ Span $\{(v_{10},v_{20})\}$. Let $\frac{w_{10}}{v_{10}}=\frac{w_{20}}{v_{20}}=c$.  Then the fact $\int_{\Omega}w_1v_{10}+w_2v_{20}=0$ implies $c=0$, and hence, $w_{10}=w_{20}=0$.
 	
 	 Let $(h_1,h_2,r)\in F$. Then consider equations
 	 \begin{eqnarray}\label{eq8}
 	 &&d_1\Delta w_1+(m-w^*-\alpha-\lambda_2(\beta_0))w_1+\beta_0w_2-tv_{10}=h_1,\nonumber\\
 	 &&d_2\Delta w_2+(m-w^*-\beta_0-\lambda_2(\beta_0))w_2+\alpha w_1-tv_{20}=h_2,\\
 	 &&2\int_{\Omega}w_1v_{10}+w_2v_{20}=r, \frac{\partial w_i}{\partial n}=0, i=1,2.\nonumber
 	 \end{eqnarray}
 	For simplicity, we use the inner product  $\langle\phi,\psi\rangle=\int_{\Omega}\phi^T\psi$. Denote $w=(w_1,w_2)^T$ and $v_0=(v_{10},v_{20})^T$. Then solving  \eqref{eq8} is equivalent to solving the inhomogeneous equation 
 	 $Lw=G$ with zero Neumann condition and the constraint $2\langle w,v_0\rangle=r$,
 	 where the self-adjoint operator  $L:C^{2+\alpha}(\bar{\Omega},\R^2)\to C^{\alpha}(\bar{\Omega},\R^2)$ s given by      $$L\left(\begin{array}{c}
     \phi_1\\\phi_2
     \end{array}\right)=\left(\begin{array}{c}
     \alpha d_1\Delta \phi_1+\alpha[m-w^*-\alpha-\lambda_2(\beta_0)]\phi_1+ \alpha\beta_0 \phi_2\\
     \beta_0 d_2\Delta \phi_2+\alpha\beta_0 \phi_1+\beta_0[m-w^*-\beta_0-\lambda_2(\beta_0)]\phi_2
     \end{array}\right),$$
 	and $G$ is given by 
 	$$G=\left(\begin{array}{c}
 	\alpha t v_{10}+\alpha h_1\\
 	\beta_0 t v_{20}+\beta_0 h_2
 	\end{array}\right).$$
Since the solution set of homogeneous equation $Lw=0$ is Span$\{v_0\}$,  the solvability criterion for $Lw=G$ is
 	$$\langle G,v_0\rangle=\langle Lw,v_0\rangle=\langle w,Lv_0\rangle=0.$$
 	A simple calculation shows that $t=-\frac{\int_{\Omega}\alpha v_{10}h_1+\beta_0v_{20}h_2}{\int_{\Omega}\alpha v^2_{10}+\beta_0v^2_{20}}$. Moreover, solutions of $Lw=G$ with zero Neumann boundary conditions can  be written in the form $z+kv_0$, where $k$ is an arbitrary constant and $z$ is uniquely determined by the requirement 
		 $\langle z,v_0\rangle=0$. Now choose $k=\frac{r}{2}$. Then $w=z+kv_0$ is a solution of \eqref{eq8}.  It then follows from implicit function theorem that $\lambda_2(\beta)$ and
 		 $(\phi_1(\beta),\phi_2(\beta))$ are continuously differentiable in $\beta$.
 		 
 		 Taking the derivative with respect to $\beta$ in \eqref{EP7} (or eqivalently in $L\phi=0$), we obtain a system equivalent to 
 		 $L\tilde{\phi}=f$ with $\tilde{\phi}=(\phi_1'(\beta),\phi_2' (\beta))^T$,

		  and $f=(\lambda'_2(\beta)\alpha\phi_1-\alpha\phi_2,\lambda'_2(\beta)\beta\phi_2+\beta\phi_2)^T$. A simple computation shows that
 		 $$0=\langle L\phi,\tilde{\phi}\rangle=\langle\phi,L\tilde{\phi}\rangle=\langle \phi,f\rangle, $$
 		 that is, $\displaystyle\lambda'_2(\beta)=\frac{\int_\Omega \alpha\phi_1\phi_2-\beta\phi^2_2}{\int_{\Omega}\alpha\phi^2_1+\beta\phi^2_2}$. Since $\lambda(d_2,m-w^*)<\lambda(d_3,m-w^*)=0$, we see from the second equation of \eqref{EP7} that $\lambda_2(\beta)\int_{\Omega}\phi^2_2-\int_{\Omega}(\alpha\phi_1\phi_2-\beta\phi^2_2)=-d_2\int_\Omega|\triangledown \phi_2|^2+\int_{\Omega}(m-w^*)\phi_2^{2}<0$. This shows if $\lambda_2(\beta)\ge0$, then $\lambda'_2(\beta)>0$. Moreover, $\lambda_2(\beta)$ changes sign at most once. 
 		  	
 		By essentially the same argument as in Lemma \ref{w}, we obtain that when $\beta\ge\frac{d_2-d_3}{d_3-d_1}\alpha$, $\lambda_2(\beta)>0$.    Suppose that $\lambda_2(\beta)$ doesn't change sign, then $\lambda_2(\beta)$ is bounded from below by zero, and there exists $\beta_n>0$ and $(\phi_{1n},\phi_{2n})$ with $\int_\Omega{\phi^2_{1n}+\phi_{2n}^2}=1$ such that $\lambda_2(\beta_n)\to A\ge0$ as $\beta_n\to0$. One may use the elliptic regularity to assume that $(\phi_{1n},\phi_{2n})\to (\hat \phi_1,\hat \phi_2)$ in $C^2(\bar\Omega)$ satisfying
 		 \begin{equation}\label{EP9}
 		 \aligned
 		 &A \hat\phi_1=d_1\Delta \hat\phi_1+(m(x)-w^*-\alpha)\hat\phi_1&&\mbox{in}\,\,\Omega,\\
 		 &A \hat \phi_2=d_2 \Delta \hat\phi_2+\alpha\hat\phi_1+(m(x)-w^*)\hat\phi_2&&\mbox{in}\,\,\Omega,\\
 		 & \frac{\partial \hat\phi_1}{\partial n}=\frac{\partial \hat\phi_2}{\partial n}=0&&\mbox{on}\,\,\partial\Omega.
 		 \endaligned
 		 \end{equation} 
 		 Since $\hat{\phi}_i\ge0, i=1,2,$ and $\int_\Omega{\hat\phi^2_{1}+\hat\phi_{2}^2}=1$, $A$ is either $\lambda(d_1,m-w^*-\alpha)<\lambda(d_1,m-\alpha)\le0$ or $\lambda(d_2,m-w^*)<\lambda(d_3,m-w^*)=0$, and hence, $A<0$, a contradiction. Statement (i) holds true.
 		 
 	 For statement (ii), we claim there exists some $\epsilon>0$ such that $\lambda(d_3,m-u^*(\beta)-v^*(\beta))>0$ if $\beta\in(0,\epsilon)$.
 	 If it is not true, then there exists $\beta_n\to 0(n\to\infty)$, $\lambda(d_3,m-u^*_n-v^*_n)\le0$ and $(u^*_n,v^*_n)\in(0,\beta_n)\times(0,\alpha)$. Since $\lambda(d_3,m-u^*_n-v^*_n)$  depends continuously  on $m-u^*_n-v^*_n$, we might assume that (up to a subsequence if necessary) $(u^*_n,v^*_n)\to (0,v^*_\infty)$ in $C^2(\bar \Omega)$ satisfying  $d_2 \Delta v^*_\infty+(m-v^*)v^*_\infty=0$ and $\lambda(d_3,m-v^*_n-u^*_n)\to \lambda(d_3,m-v^*_\infty)\le0$. 
 	If $v^*_\infty\equiv0$, then $\lambda(d_3,m)>0$ due to assumption (H), a contradiction. Otherwise, $v^*_\infty$ is positive, $m-v^*_{\infty}$ is non-constant and $\lambda(d_2,m-v^*_\infty)=0<\lambda(d_3,m-v^*_\infty)$, a contradiction again.
 	
 	In the case that $\beta\ge\frac{d_2-d_3}{d_3-d_1}\alpha$, it follows directly from Lemma \ref{uv} that $\lambda(d_3,m-u^*-v^*)<0$, that is, $(u^*,v^*,0)$ is linearly stable.
 	
 	For statement (iii), we show that when $\beta\to0^+$, there is no coexistence steady state. If not, then there exists $\beta_n\to 0(n\to\infty)$,  positive steady states $(u^0_n,v^0_n,w^0_n)$ and $(u^0_n,v^0_n)\in(0,\beta_n)\times(0,\alpha)$. Passing to the limit, we might assume that (up to a subsequence if necessary) $(u^0_n,v^0_n,w^0_n)\to(0,v^0_\infty,w^0_\infty)$ in $C^2(\bar{\Omega})$ satisfying
 	\begin{equation}\label{eq9}
 	\aligned
 	&0=d_2 \Delta v^0_\infty+v^0_\infty(m(x)-v^0_\infty-w^0_\infty)&&\mbox{in}\,\,\Omega,\\
 	&0=d_3 \Delta w^0_\infty+w^0_\infty(m(x)-v^0_\infty-w^0_\infty)&&\mbox{in}\,\,\Omega,\\
 	& \frac{\partial v^0_\infty}{\partial n}=\frac{\partial w^0_\infty}{\partial n}=0&&\mbox{on}\,\,\partial\Omega.
 	\endaligned
 	\end{equation}
 	Clearly, $v^0_\infty,w^0_\infty$ can not be both positive. There will be three possible cases, that is, (a) $(v^0_\infty,w^0_\infty)=(0,0)$; (b) $(v^0_\infty,w^0_\infty)=(0,w^*(d_3))$; (c) $(v^0_\infty,w^0_\infty)=(v(d_2),0)$. However,  essentially the same proof as in \cite[Lemma 4.5]{ZZ}  implies that none of them can happen.  Suppose that case (a) occurs.  Let 	
	$\widehat{v}_n=\displaystyle \frac{v^0_n}{||v^0_n||_{L^\infty}}$.  We have
	$$0=d_2\Delta\widehat{v}_n +(m(x)-u^0_n+v^0_n-w^0_n)\widehat{v}_n.$$
	
	We can assume, by passing to a subsequence if necessary, that $\widehat{v}_n \to \widehat{v}^*$ with 
	$||\widehat{v}^*||_{L^\infty}=1$, where $\widehat{v}^*$ satisfies $0=d_2\Delta\widehat{v}^* +(m(x)-w^*(d_3))\widehat{v}^*$.   However, $0=d_3\Delta w^*(d_3)+(m(x)-w^*(d_3))w^*(d_3)$, so the principal eigenvalue of the operator $Lw=d_3\Delta w +(m(x)-w^*(d_3))w$ is $0$, so the strict monotonicity of the principal eigenvalue with respect to the diffusion coefficient gives a contradiction to $0=d_2\Delta\widehat{v}^* +(m(x)-w^*(d_3))\widehat{v}^*$.  The argument for case (b) is very similar so we omit it. 
In case (c) we use $\widehat{w}_n=\displaystyle \frac{w^0_n}{||w^0_n||_{L^\infty}}$ and pass to  a limit $\widehat{w}^*$. An argument analogous to the one given previously for 
$\widehat{v}^*$ leads to the equation $0=d_3\Delta \widehat{w}^*+(m(x)-v(d_2))\widehat{w}^*$ with $||\widehat{w}^*||_{L^\infty}=1$, but since $0=d_2\Delta v(d_2)+(m(x)-v(d_2))\widehat{w}^*v(d_2)$ and $d_3 \neq d_2$ this also leads to a contradiction. Hence none of the cases (a),(b), or (c) is possible, so there cannot be a coexistence state as $\beta\to0^+$.
 	Following the same idea as in the proof of Lemma \ref{no}, we can see that $\beta\to\beta^-_0$ with $\beta_0=\frac{d_2-d_3}{d_3-d_1}\alpha$, there is no coexistence steady state. \end{proof}
A parallel result is stated below when $\alpha$ varies.
 	\begin{lemma} \label{alpha}Assume that $d_1<d_3<d_2$ and $\max_{x\in\bar\Omega}m(x)\le \beta$. Then the following statements are valid.
 		\begin{enumerate}
 			\item [(i)]There exists a unique $\alpha_c\in(\frac{d_3-d_1}{d_2-d_3}\beta,\infty)$, such that $(0,0,w^*)$ is linearly unstable when $\alpha\in(0,\alpha_c)$; linearly stable when $\alpha\in(\alpha_c,\infty)$.
 			\item[(ii)] $(u^*,v^*,0)$ is linearly stable when $\alpha$ is small enough; linearly unstable when $\alpha\in[\frac{d_3-d_1}{d_2-d_3}\beta,\infty)$.
 			\item[(iii)]There exists small $\epsilon>0$, such that system \eqref{PDEModel2} admits no positive steady state when $\alpha\in(0,\epsilon)\cup (\frac{d_3-d_1}{d_2-d_3}\beta-\epsilon,\infty)$.  
 		\end{enumerate}  
 	\end{lemma} 
 \begin{proof}
 	We only prove Statement (i) when $\alpha$ is large, as the other cases are analogous to the proof in Lemma \ref{beta}.
 	By an argument similar to those in Lemma \ref{beta}, we have $\lambda_2(\alpha)$ is continuously differentiable in $\alpha>0$, and $\lambda(\alpha)>0$ when $\alpha\le\frac{d_3-d_1}{d_2-d_3}\beta$. A  direct computation shows that $\displaystyle\lambda'_2(\alpha)=\frac{\int_\Omega \beta\phi_1\phi_2-\alpha\phi^2_1}{\int_{\Omega}\alpha\phi^2_1+\beta\phi^2_2}$.
 	If $\lambda_2(\alpha)=0$ for some $\alpha>0$, then we see from \eqref{EP7} and \eqref{Ineq3} that $\int_{\Omega}\alpha \phi^2_1-\beta\phi_1\phi_2=-d_1\int_\Omega|\triangledown \phi_1|^2+\int_{\Omega}(m-w^*)\phi_1^{2}>0$. Therefore, $\lambda'_2(\alpha)<0$ when $\lambda_2(\alpha)=0$. This implies that $\lambda_2(\alpha)$ can change signs at most once.
 	Suppose $\lambda_2(\alpha)$ does not change signs; that is, $\lambda_2(\alpha)>0,\,\forall \alpha>0$. Since $\int_{\Omega} \beta\phi^2_2-\alpha\phi_1\phi_2+\lambda_2\int_\Omega{\phi_2^2}=-d_2\int_\Omega|\triangledown \phi_2|^2+\int_{\Omega}(m-w^*)\phi_2^{2}<0$ and $$\alpha[\int_{\Omega}\alpha \phi^2_1-\beta\phi_1\phi_2] +\beta [\int_{\Omega} \beta\phi^2_2-\alpha\phi_1\phi_2]\ge0,$$
 we have  $\int_{\Omega}\alpha \phi^2_1-\beta\phi_1\phi_2>0$, so is $\lambda'(\alpha)<0$.  Hence
 	 $\lambda_2(\alpha)$ is strictly decreasing in $\alpha>0$ and  uniformly bounded from below. Let  $(\phi_{1\alpha},\phi_{2\alpha})\in C^{2,\nu}(\bar\Omega,\R^2_+)$ with $\int_{\Omega}\alpha\phi_{1\alpha}^2+\beta\phi^2_{2\alpha}=1$ be the associated eigenfunction with $\lambda_2(\alpha)$.  Then $\lambda_2(\alpha)\to\lambda_{\infty}\ge0$,  as $\alpha\to\infty$. 
 
A straightforward calculation indicates that
\begin{align*}\int_\Omega\phi^2_{1\alpha}&=\frac{-d_1\int_\Omega|\triangledown \phi_{1\alpha}|^2+\int_{\Omega}(m-w^*-\lambda_2)\phi_{1\alpha}^{2}+\beta\phi_{1\alpha}\phi_{2\alpha}}{\alpha}\\
&\le \frac{\bar m+\beta}{\alpha}\to 0,\quad \alpha\to\infty,
\end{align*}
and $$0<\int_{\Omega}\alpha\phi^2_{1\alpha}-\beta\phi_{1\alpha}\phi_{2\alpha}= -d_1\int_\Omega|\triangledown \phi_{1\alpha}|^2+\int_{\Omega}(m-w^*-\lambda_2)\phi_{1\alpha}^{2}   \le \bar{m} \int_\Omega\phi_{1\alpha}^{2} $$ 	
which yields that $\lim\limits_{\alpha\to\infty}\int_\Omega\alpha\phi^2_{1\alpha}=0$ and $\lim\limits_{\alpha\to\infty}\int_{\Omega}\alpha |\triangledown \phi_{1\alpha}|^2=0$ due to the fact from \eqref{Ineq3} that $-\alpha d_1\int_\Omega|\triangledown \phi_{1\alpha}|^2+\alpha\int_{\Omega}(m-w^*)\phi_{1\alpha}^{2}>0$.  In view of identity \eqref{Ineq3} again, we have
$$\int_\Omega(\alpha\phi_{1\alpha}-\beta\phi_{2\alpha})^2\le \int_{\Omega}(m-w^*)\alpha\phi^2_{1\alpha}\le \bar{m}\int_{\Omega}\alpha\phi^2_{1\alpha}\to0, \quad \alpha\to \infty.$$
Therefore, $\|\alpha\phi_{1\alpha}-\beta\phi_{2\alpha}\|_{L^2(\Omega)}\to 0$ as $\alpha\to\infty$. 

Now replace $\lambda_2$ and $(\phi_{1},\phi_2)$ by $\lambda_2(\alpha)$ and $(\phi_{1\alpha},\phi_{2\alpha})$ in \eqref{Ineq3} and let $\alpha\to\infty$. Then we obtain
$$\lim_{\alpha\to\infty}-d_2\int_\Omega|\triangledown \phi_{2\alpha}|^2+\int_{\Omega}(m-w^*)\phi_{2\alpha}^{2}=\lambda_\infty\ge0.$$
Indeed, $\lambda_{\infty}=0$ due to the fact that $-d_2\int_\Omega|\triangledown \phi_{2\alpha}|^2+\int_{\Omega}(m-w^*)\phi_{2\alpha}^{2}<0$ for any $\alpha>0$. 
Now we see that $\phi_{2\alpha}$ is bounded in $H^1(\Omega)$ when $\alpha$ is large. This implies (up to a subsequence if necessary) $\phi_{2\alpha}\to \phi_{\infty}$ in $L^2(\Omega)$ with $\|\sqrt{\beta}\phi_{\infty}\|_{L^2(\Omega)}=1$.  Let $A\phi:= -d_2\Delta\phi-(m-C)\phi$ for some large $C>\lambda(d_2,m)$. Then $A^{-1}:L^2(\Omega)\to H^2(\Omega)$ is a continuous operator. Passing to the limit in  $\phi_{2\alpha}=A^{-1}[\alpha\phi_{1\alpha}-\beta\phi_{2\alpha}+(C-\lambda_2(\alpha))\phi_{2\alpha}]$, we get $\phi_{\infty} =A^{-1}((C-\beta)\phi_{\infty})$.   Standard elliptic regularity	implies $\phi_{\infty} \in C^{1,\nu}(\bar\Omega)$, and hence, $-d_2\int_\Omega|\triangledown \phi_{\infty}|^2+\int_{\Omega}(m-w^*)\phi_{\infty}^{2}=0$ and $\phi_{\infty}\not\equiv0$, that is, $0\le\lambda(d_2,m-w^*)<\lambda(d_3,m-w^*)=0$, a contradiction. 

It follows immediately that $\lambda_2(\alpha)$ changes sign once and has a unique $\alpha_c\in(0,\frac{d_3-d_1}{d_2-d_3}\beta)$ such that $\lambda_2(\alpha_c)=0$.\end{proof}

Now we are ready to state two parallel results on the global dynamics of the boundary steady state in terms of $\alpha$ and $\beta$, respectively. They follow from the same arguments based on monotone dynamical systems that are used in Theorem 3.6.
 \begin{theorem}
 	Assume that $d_1<d_3<d_2$ and $\max_{x\in\bar\Omega}m(x)\le\alpha$. Then there exist some $0<C_1\le C_2<\frac{d_2-d_3}{d_3-d_1}\alpha$ such that the following statements are valid.
 	\begin{enumerate}
 		\item[(i)] $(0,0,w^*)$ is globally asymptotically stable in $X^+_1\times (X^+_2\setminus\{0\})$ when $\beta\in(0,C_1)$.
 		\item[(ii)] $(u^*,v^*,0)$ is globally asymptotically stable in $(X^+_1\setminus\{0\})\times X^+_2$ when $\beta\in(C_2,\infty)$.
 		 
 	\end{enumerate}
 \end{theorem}
 \begin{theorem}
 	Assume that $d_1<d_3<d_2$ and $\max_{x\in\bar\Omega}m(x)\le\beta$. Then there exist some $0<C_1\le C_2<\frac{d_3-d_1}{d_2-d_3}\beta$ such that the following statements are valid.
 	\begin{enumerate}
 		\item[(i)] $(0,0,w^*)$ is globally asymptotically stable in $X^+_1\times (X^+_2\setminus\{0\})$ when $\alpha\in(C_2,\infty)$.
 		\item[(ii)] $(u^*,v^*,0)$ is globally asymptotically stable  in $(X^+_1\setminus\{0\})\times X^+_2$ when  $\alpha\in(0,C_1)$. 
 	\end{enumerate}
 \end{theorem}
 
  \section{Conclusions}

 For the two-component subsystem \eqref{PDEModel} we have derived conditions under which it is asymptotically competitive or cooperative.  In the asymptotically cooperative case we have derived further conditions implying the existence of a unique globally stable positive equilibrium. We have obtained various eigenvalue estimates that determine the stability of the trivial solution $(0,0)$.  Some of the results for \eqref{PDEModel} are extensions of those in \cite{CCY} to cases where some coefficients may vary in $x$.   We should note that we have not been able to give a complete analysis of the stability of $(0,0)$ in the indefinite case, that is, where the local population growth rate $m(x)$ can change sign, reflecting  the presence of both sources and sinks in the overall environment.  This is due to the fact that we do not know of an extension of a key result from \cite{HK} to systems of equations.  A major reason why results implying that the sub-model \eqref{PDEModel} has a unique globally attracting positive equilibrium are interesting is that in such a case  \eqref{PDEModel} behaves like a single logistic equation and hence it is reasonable to view the  populations  $u$ and $v$ together as a single population consisting of individuals that can switch their dispersal behavior.  \\
  
  The main problem motivating this paper  was that of understanding how well a population whose members can switch between slow and fast diffusion rates $d_1$ and $d_2$ could compete against an ecologically identical population   where all individuals diffuse at a single intermediate rate $d_3$.   What we found was that if $d_3< d_1<d_2$ then the semi trivial equilibrium 
  $(u^*,v^*,0)$  of \eqref{PDEModel2} is unstable and $(0,0,w^*)$ is stable, while if $d_1<(\alpha d_1+\beta d_2)/(\alpha+\beta)<d_3$ then $(u^*,v^*,0)$  stable and $(0,0,w^*)$ is unstable.  Furthermore, both semi trivial equilibria change their stability for  some values of $d_3$ in the interval $(d_1,(\alpha d_1+\beta d_2)/(\alpha+\beta))$.  Thus, the size of the diffusion rate $d_3$ relative to the average of diffusion rates $d_1$ and $d_2$ weighted by the switching rates $\alpha$ and $\beta$ seems to be informative about which of the populations $(u,v)$ and $w$ has the advantage.  In some cases we were able to show the nonexistence of a positive (coexistence) equilibrium for \eqref{PDEModel2}, which then implies competitive exclusion when combined with suitable results  on stability of semi trivial equilibria.  \\
  
  There remain many challenging open questions about \eqref{PDEModel2} and related models. In the case where $m(x)$ changes sign, we do not have a uniqueness result for the principal eigenvalue of the linearized model corresponding to the sub model \eqref{PDEModel}.   Since we can show that the semi trivial equilibria can change stability as $d_3$ or $\alpha$ or $\beta$ vary we expect that the system \eqref{PDEModel2}  will have bifurcations that produce coexistence states (which might be unstable), but we have not explored a bifurcation theoretic approach, and we do not have enough information about the relative locations relative to $d_3$ of the points where the stabilities of $(u^*,v^*,0)$ and $(0,0,w^*)$ change to use monotone methods to show the presence of coexistence states. It should be possible to address these and other questions but that will require additional research.  In a different direction, it would be interesting to consider models with different types of dipsersal operators, boundary conditions, or interaction terms.  Another topic of interest would be to try to see if and when adaptive switching that mimics area restricted search (that is, switching that is biased toward slower diffusion at locations where $m(x)$ is large, but toward faster diffusion where  $m(x)$ is small) is advantageous versus diffusion at a fixed rate everywhere. Some numerical results about this type of phenomenon in a more realistic dispersal model are given in \cite{Fagan}.  In general, the idea that  organisms switch between different movement modes has considerable empirical support and leads to mathematical models whose analysis is challenging but within the scope of current mathematical methods. For those reasons we think dispersal models with switching are an interesting topic for further study.

\end{document}